\documentclass[a4paper,11pt]{article}

\usepackage[utf8x]{inputenc}
\usepackage[mathcal]{euscript}
\usepackage{hyperref}
\usepackage{pdfpages}
\usepackage[brazil,portuguese, english]{babel}
\usepackage{amsmath}
\usepackage{dsfont}
\usepackage{amssymb}
\usepackage{amsthm}
\usepackage{enumitem}
\usepackage{xcolor}
\usepackage{graphicx}
\usepackage{tikz}
\usepackage[disable]{todonotes}

\newcounter{example}
\newenvironment{example}[1]{\refstepcounter{example}\textbf{Example \theexample :} #1}{\medskip}
\newcommand{\BA}{\mathbb{A}}
\newcommand{\BB}{\mathbb{B}}
\newcommand{\BD}{\mathbb{D}}
\newcommand{\BE}{\mathbb{E}}
\newcommand{\BF}{\mathbb{F}}
\newcommand{\BI}{\mathbb{I}}
\newcommand{\BK}{\mathbb{K}}

\newcommand{\BP}{\mathbb{P}}
\newcommand{\BR}{\mathbb{R}}
\newcommand{\BS}{\mathbb{S}}
\newcommand{\BZ}{\mathbb{Z}}

\newcommand{\cc}{\mathfrak{c}}
\newcommand{\cT}{\mathcal{T}}
\newcommand{\supp}{\rm supp}
\newcommand{\Sep}{{\rm Sep}}

\newcommand{\bbR}{\mathbb{R}}
\newcommand{\bbZ}{\mathbb{Z}}
\renewcommand{\L}{\mathcal{L}}
\newcommand{\md}[1]{\mathrm{\mu}_{d,1}\left(#1\right)}
\newcommand{\cyn}[1]{\mathfrak{c}\left(#1\right)}
\newcommand{\cC}{\mathcal{C}}

\newcommand{\di}[2]{\mathop{d}\left(#1,#2\right)}
\newcommand{\abs}[1]{\left| #1 \right|}
\newcommand{\p}[1]{\mathbb{P}\left( #1 \right)}
\newcommand{\pc}[2]{\mathbb{P}\left( #1 | #2 \right)}
\newcommand{\expo}[1]{\mathop{\exp}\left(#1\right)}
\newcommand{\E}[1]{\mathbb{E}\left(#1\right)}
\newtheorem{theo}{Theorem}[section]
\newtheorem{defi}[theo]{Definition}

\newtheorem{lemma}[theo]{Lemma}

\newtheorem{prop}[theo]{Proposition}

\newcounter{assumptions}

\title{Random cover times using the Poisson cylinder process}
\author{Erik I. Broman, \footnote{Chalmers Univeristy of Technology
and Gothenburg University, email: broman@chalmers.se } 
Filipe Mussini \footnote{Uppsala University, email: filipe.mussini@math.uu.se}}

\begin{document}

\maketitle

\begin{abstract}
In this paper we deal with the classical problem of random cover times. 
We investigate the distribution of the time it takes
for a Poisson process of cylinders to cover a set $A \subset \BR^d.$
This Poisson process of cylinders is invariant under rotations, reflections 
and translations, and in addition we add a time component so that 
cylinders are ''raining from the sky'' at unit rate.
Our main results concerns the asymptotic of this cover time as the set 
$A$ grows. If the set $A$ is discrete and well separated, we show
convergence of the cover time to a Gumbel distribution. If instead 
$A$ has positive box dimension (and satisfies a weak additional assumption), 
we find the correct rate of convergence. 
\end{abstract}

\section{Introduction}\label{sec:Introduction}
There are many variants of coverage problems that has been studied in the 
probabilistic literature. One of the first papers on this subject was 
by Dvoretzky (\cite{Dvor}) and dealt with the problem of covering the circle 
by using a sequence of sets placed randomly around the circle. The related
problem of covering $\BR^d$ was then later studied by Shepp in (\cite{Shepp})
for $d=1$,  Bierm\'{e} and Estrade (\cite{BE}) for general $d$ and also
by Broman, Jonasson and Tykesson (\cite{BJT}) for general $d.$  
In common to all of these papers is that there were no time-component involved. 
Instead, they all consider (infinite) measures $\mu$ on the set of compact subsets
of $\BR^d$. Then, they study a Poisson process using $\mu$ as the intensity 
measure, and ask whether $\BR^d$ will be completely covered. Of course, this 
will depend on the particular choice of $\mu$ and this dependence is what they
investigate.

A variant of this covering problem is to do the following. Start with some
bounded set $A \subset \BR^d$, and throw down other (possibly random) sets 
$B_i$ in random locations, and proceed until $A$ is covered. 
It is then natural to ask about the distribution of the number of sets
needed in order to cover $A$. Alternatively, if the sets are dropped at 
unit rate, one can ask about the distribution of the cover time, i.e. the 
time it takes until $A$ is covered.
An example would 
be to let $A$ be the unit square in $\BR^2,$ and letting $B_i$ all be squares
of side length $\epsilon$ with their centres uniformly distributed in $A.$ 
Another example would be to let the side length of the sets $B_i$ be random.
Such problems was studied by for instance Siegel and Holst (\cite{SH}) 
and Janson \cite{Jan1} on the circle, while a much more general 
result was later obtained by Janson in \cite{Jan2}. In particular,
all of these papers studied asymptotics of the cover time as the set
$A$ increases.

More recently, Belius \cite{Belius} studied the problem of covering a bounded set 
$A \subset \BZ^d$ by what is known as random interlacements. This is basically
a Poisson process on the trajectories of bi-infinite random walks in $\BZ^d$
and was introduced by Sznitman in \cite{Sznit}. The major difference between
the paper by Belius and the others mentioned above, is that the interlacement 
trajectories are unbounded objects, whereas in the classical setting, the 
corresponding sets are finite. The use of infinite objects introduces a number 
of new challenges as the cover times of separated sets are no longer independent. 

\medskip

The aim of this paper is to combine the classical problem of covering sets
$A \subset \BR^d$ of non-zero dimension (rather than say a subset of $\BZ^d$), 
with the use of unbounded objects to cover the set.  
In order to explain our main result, we shall here give informal 
descriptions of the
mathematical quantities and tools needed. Precise definitions and explanations
will be given in Section \ref{sec:models}. For the models to make sense, 
we will always assume that the dimension $d$ is at least $2$, and this will be 
assumed throughout the paper without any further comment.

We will use a Poisson 
process $\Psi$ where an element $(L,s)\in \Psi$ consists of a
line $L \subset \BR^d$ and a ''time-stamp'' $s \in \BR^+.$ 
The set of lines with time-stamp
smaller than $t$ will then be a Poisson process on the set of 
lines in $\BR^d$, and this process will be invariant under rotations, 
reflections and translations (see Section \ref{sec:models}). 
For such a line $L$ we consider the corresponding
cylinder $\cc(L)$ with base radius $1,$ i.e.
\[
\cc(L) := \{ x \in  \bbR^d : \di{x}{L} \leq 1 \}. 
\]
Then we ask whether
\[
A \subset \bigcup_{(L,s)\in \Psi:s \leq t} \cc(L)?
\] 
That is, we ask whether $A$ has been covered by the cylinders $\cc(L)$
which have dropped before time $t,$ and we let the {\em cover time} of 
$A$ be 
\[
\cT(A):=\inf\left\{t > 0:A \subset \bigcup_{(L,s)\in \Psi:s \leq t} \cc(L)\right\}.
\]

In order to state our main result, we will need to define the box dimension of 
a bounded set $A \subset \BR^d.$ A further albeit brief discussion of 
this concept is also included in Section \ref{sec:dim}. 
Let $N_\delta(A)$ be the minimum number of boxes of side length $\delta>0$
needed to cover $A.$ The box dimension of $A$ is then defined by 
\[
\dim_B(A)=\lim_{\delta \to 0} \frac{N_\delta(A)}{-\log \delta},
\]
whenever this limit exists. 

We can now state our main theorem.

\begin{theo} \label{thm:main}
For any set $A \subset \BR^d$ such that $\dim_B(A)$ is well defined, 
and which satisfies the additional condition 
\begin{equation} \label{eqn:dimliminfsup}
0<\liminf_{\delta \to 0} \delta^{\dim_B(A)} N_\delta(A)
\leq \limsup_{\delta \to 0} \delta^{\dim_B(A)} N_\delta(A)<\infty,
\end{equation}
we have that the sequence 
\[
(\cT(nA)-\dim_B(A)(\log n+\log \log n))_{n \geq 1}
\]
is tight.
\end{theo}

\noindent
{\bf Remarks:} Of course, $nA$ is simply the set $A$ scaled 
by a factor of $n.$

It is reasonable (in light of similar results such 
as the ones in \cite{Jan2}), to expect that if $A$ satisfies a strong
enough regularity condition, 
\[
\cT(nA)-\dim_B(A)(\log n+\log \log n) +C
\]
converges to a Gumbel distribution for a suitable choice of constant $C.$
Our main result does not quite achieve this.
However, it does show that if 
$\cT(nA)-f(n)$ converges to a non-trivial random variable, then the
function $f(n)$ must in fact take the form 
$\dim_B(A)(\log n+\log \log n) +O(1)$. 
See also the remark after the proof of Theorem \ref{thm:main}.

It might not be clear whether a ''typical'' set should satisfy condition
\eqref{eqn:dimliminfsup}.
While we will not discuss this question in full detail (as it is more 
a statement that belongs to fractal geometry), we give a basic result 
(Proposition \ref{prop:Arhoex})
and provide a few examples in Section \ref{sec:applications}. That section
should be sufficient motivation that assumption \eqref{eqn:dimliminfsup} 
is satisfied for a rich family of sets.

In Theorem \ref{thm:aux}, we will provide an auxiliary (and somewhat weaker) 
result that covers the case when \eqref{eqn:dimliminfsup} fails.

\bigskip

Our second main result deals with sequences of finite subsets of $\BR^d.$
As this text deals with both finite subsets of $\BR^d$ and non-discrete 
subsets, it will be convenient to use different notations for the two. 
Therefore we let $\BA,\BD$ etc denote finite subsets of 
$\BR^d,$ while $A\subset \BR^d$ will be used to denote a non-discrete 
set. We will use the following definition
\[
\Sep(\BD):=\inf\{d(x,y): x,y\in \BD\},
\]
so that $\Sep(\BD)$ denotes the minimum separation distance between any 
two points in the set $\BD.$ A finite set $\BD$ with $\Sep(\BD) \geq \rho$
will be called $\rho$-separated.

Here and in the rest of the paper, $|\cdot|$ denotes cardinality. 

\begin{theo} \label{thm:maindiscrete}
Let $(\BD_n)_{n\geq 1}$ be a sequence of finite subsets of $\BR^d$ such 
that $\lim_{n \to \infty} |\BD_n|=\infty$ and such that 
\[
\liminf_{n} \Sep(\BD_n) \log |\BD_n|=\infty.
\]
Then, $(\cT(\BD_n) -\log\abs{\BD_n})_{n \geq 1}$ converges in 
distribution to the 
Gumbel distribution as $n \to \infty$.
\end{theo}
\noindent
{\bf Remark:}
In Section \ref{sec:fluct}, we will prove Theorem \ref{thm:fluct}, 
which provides a bound for the fluctuation of 
$\BP((\cT(\BD)-\log |\BD|\leq z)$ from the Gumbel distribution function.
This theorem will then readily imply Theorem \ref{thm:maindiscrete}. 

The same fluctuation result will also be used later to prove 
Theorem \ref{thm:main}. In order to give an informal explanation of 
how this is done, 
consider Theorem \ref{thm:main} and some set $A$. 
We will consider a carefully chosen
finite set sitting inside of $nA$, and the cover time of this
set will provide a lower bound to $\cT(nA).$ In order to
find an upper bound on $\cT(nA),$ we will consider the same discrete set, 
but require that each point is singularly covered (meaning that a small ball 
around the point is covered by a single cylinder of the process). 
If all points are singularly covered, this will imply that the set $nA$
is covered, and this allows us to obtain an upper bound to the 
cover time of $nA.$

\bigskip

Before we wrap up this section, 
we want to provide a short intuitive explanation of the overall strategy 
of proving our main results. Informally, if the cover time is $\tau,$ 
then immediately before this time (i.e. at time $\tau(1-\epsilon)$), 
the set which is uncovered will with high probability consist
of small islands which are well separated. Indeed, this is the case 
as Proposition \ref{prop:Gbound} shows. In addition, 
Proposition \ref{prop:Gbound} tells us that the number of 
such islands will be highly concentrated around its expected value.
This can then be combined with Proposition \ref{prop:almostindependentcover} 
which tells us that the cover times of these well separated islands are 
almost independent. These two results are then combined into 
Theorem \ref{thm:fluct}, which provides the final estimates of the 
fluctuations of the cover time $\cT(\BD).$ This result is then in turn 
used to prove both Theorem \ref{thm:main} and Theorem \ref{thm:maindiscrete}.

\bigskip

The structure of the rest of the paper is as follows: In Section 
\ref{sec:models} we define the Poisson cylinder model and briefly 
discuss box dimensions. In Section \ref{sec:Cylprel} we present some 
preliminary results concerning the 
Poisson cylinder model, while in Section \ref{sec:packing} we prove some 
fundamental properties about $\rho$-separated sets. Then, 
Section \ref{sec:fluct}
is devoted to the proof of the fluctuation result, i.e. Theorem \ref{thm:fluct}
which is used to prove Theorem \ref{thm:maindiscrete} in the same section. 
The overall strategy in this section is similar to the strategy of \cite{Belius}
and described above.
However, there will also be many differences stemming from 
the fact that we are working in $\BR^d$ as opposed to the discrete space 
$\BZ^d$. Section \ref{sec:mainproof} is dedicated to proving 
Theorem \ref{thm:main} while
Section \ref{sec:secondaryproof} is used to prove our above mentioned secondary 
result, i.e. Theorem \ref{thm:aux}. Finally, Section \ref{sec:applications} 
contains some examples and Proposition \ref{prop:Arhoex} that pertains to 
assumption \eqref{eqn:dimliminfsup} of Theorem \ref{thm:main}.

\medskip

We end this section with a comment on notation. We shall frequently 
use $c$ to denote a constant (depending only on $d$) which may 
change from line to line. In contrast, numbered constants $c_k$ will 
be fixed.

\section{Model and definitions} \label{sec:models}
This section is divided into two subsections. The first one includes 
the definition and some preliminaries of the Poisson cylinder model. 
The second subsection will 
provide some basic background on box dimensions.

\subsection{The Poisson cylinder model} \label{subsec:PCmodel}
Our first step is to define the Poisson {\em line} model. To that end, 
let $G(d,1)$ be the set of infinite lines in $\bbR^d$ that pass 
through the origin $o$, and let $A(d,1)$ be the set of infinite lines in 
$\bbR^d$. Furthermore, let
$$\L_K := \{L \in A(d,1) : L \cap K \neq \emptyset \},$$
be the set of lines that intersect $K$ where $K \subset \bbR^d$ is a compact set.
For convenience, we let $\L_{A,B}$ denote the set $\L_A \cap \L_B$, i.e. 
the set of lines that intersect both sets $A,B \subset \bbR^d.$
Furthermore, $B(x,\rho)$ will denote the closed ($d$-dimensional) 
ball of radius $\rho$ centered at $x$.

We let $\nu_{d,1}$ be the unique Haar measure on the space $G(d,1)$ 
normalized so that 
$\nu_{d,1}(G(d,1))=1$. Furthermore, on $A(d,1)$ there 
is a unique (up to constants) 
measure which is invariant under rotations, reflections and translations. 
We will let $\mu_{d,1}$ denote this latter measure, normalized so that 
$\md{\L_{B_d(o,1)}} =1$ (see for instance \cite{SchneiderWeil} Chapter 13).
For any subspace
$H\subset \bbR^d,$ and set $D\subset \bbR^d,$ we let $\Pi_{H}(D)$
denote the projection of $D$ onto $H.$ Here, we will consider
$\Pi_{H}(D)$ to be a subset of $\bbR^d$ (and not just a subset of $H$).
Furthermore, we will let $\kappa_d$ denote the volume of the 
unit ball $B(o,1)$ in $\BR^d,$
and $\lambda_d$ denote Lebesgue measure on $\bbR^d$ so that
$\kappa_d=\lambda_d(B(o,1)).$

For any $L\in G(d,1)$ we will let $L^{\bot}$ be the $(d-1)$-dimensional 
hyperplane orthogonal to $L.$ The following representation 
(\cite{SchneiderWeil} Theorem $13.2.12$) of the measure $\mu_{d,1}$ will 
be useful for us. For any $K\subset \BR^d$ we have that 
\begin{equation} \label{eqn:repform1}
\md{\L_{K}}
=\frac{1}{\kappa_{d-1}}\int_{G(d,1)} 
\int_{L^\bot} \mathds{1}(L+y\in \L_{K}) \lambda_{d-1}(dy)\nu_{d,1}(dL),
\end{equation}
where $\mathds{1}$ denotes an indicator function.
Informally, for a fixed line $L,$ the inner integral integrates over 
all lines parallel to $L$ 
that intersects $K.$ Then, the outer integral integrates over all possible 
choices of $L.$

Our next step is to consider the following space of point measures on $A(d,1)$,
\[
\Omega = \{\omega = \sum_{i=1}^\infty \delta_{L_i}: L_i \in A(d,1) \mbox{ and }
 \omega(\L_A) < \infty \mbox{ for all compact } A \subset \bbR^d\},
\]
where $\delta_L$ denotes point measure at $L$. By standard abuse 
of notation, we will sometimes identify the random measure $\omega \in \Omega$ 
with its support $\supp(\omega)$, which is really a subset of $A(d,1).$

We define $\Psi$ to be a Poisson point process on $A(d,1)\times \bbR^+$ 
with intensity measure $\mu_{d,1}\times \lambda_1^+$ where
$\lambda_1^+$ denotes Lebesgue measure on $\bbR^+$.
We think of an element $(L,s)\in \Psi$ as a line in $\bbR^d$ accompanied 
with a time-stamp $s.$ We then let $\omega_t=\Pi_{A(d,1)} \{(L,s)
\in A(d,1)\times\bbR^+ : s\leq t\}$, where $\Pi_{A(d,1)}$ denotes projection
onto the space $A(d,1).$ Thus, $\omega_t$ is the collection of
lines which have been placed before or at time $t,$ and $\omega_t \in \Omega.$
Of course, by our definition,
$\omega_t$ is in fact a Poisson process on $A(d,1)$ with intensity measure 
$t \mu_{d,1}.$ Similarly, we let $ \omega_{t_1,t_2} 
=\Pi_{A(d,1)} \{(L,s) \in A(d,1)\times\bbR^+: t_1< s\leq t_2\}$, so that 
$\omega_{t_1,t_2}$ is the set of lines placed between times $t_1$ and $t_2.$
The intensity measure of $\omega_{t_1,t_2}$ is therefore 
$(t_2-t_1)\mu_{d,1}$. Obviously we have that 
$\omega_{t_1} \cup \omega_{t_1,t_2}=\omega_{t_2}$, and we will let 
$(\omega_t)_{t \geq 0}$ denote the corresponding process.

We will sometimes slightly abuse notation by 
writing $\cyn{L} \in \omega_t$ instead of $L\in \omega_t$, and we will often 
think of (and refer to)
$\omega_t$ as a collection of cylinders instead of lines.

\subsection{Box dimensions} \label{sec:dim}
In this subsection we shall review some basic properties of 
box dimensions, sometimes referred to as Minkowski dimensions
(see Falconer \cite{Falconer} Chapter 3).

Recall the definition of $N_\delta(A)$ from the introduction. 
Then, define 
\[
\overline{\dim}_B(A)
:=\limsup_{\delta \to 0} \frac{\log N_\delta(A)}{-\log \delta},
\]
called the {\em upper box dimension} of the set $A$. 
Similarly, we define 
\[
\underline{\dim}_B(A)
:=\liminf_{\delta \to 0} \frac{\log N_\delta(A)}{-\log \delta},
\]
to be the {\em lower box dimension} of the set $A,$
and if these coincides, then 
\[
\dim_B(A)=\overline{\dim}_B(A)=\underline{\dim}_B(A)
\]
is simply called the box dimension of $A.$

The two quantities in \eqref{eqn:dimliminfsup} are 
related to the upper and lower Minkowski content (see
\cite{Federer} Sections 3.2.37-3.2.44). It is beyond the scope of this 
paper to investigate this relationship in detail. However, as mentioned
already in the remarks after the statement of Theorem \ref{thm:main},
it will be the case that \eqref{eqn:dimliminfsup} is satisfied for many 
sets. See in particular Proposition \ref{prop:Arhoex}.

We can now present Theorem \ref{thm:aux} mentioned in the remarks
after Theorem \ref{thm:main}.
\begin{theo} \label{thm:aux}
For any $\underline{\alpha}<\underline{\dim}_B(A)$, 
we have that for every $z\in \BR$, 
\[
\lim_{n \to \infty}\BP(\cT(nA)-\underline{\alpha}\log n \leq z)=0.
\]
Furthermore, for any $\overline{\alpha}>\overline{\dim}_B(A)$, we have that 
for every $z\in \BR$, 
\[
\lim_{n \to \infty}\BP(\cT(nA)-\overline{\alpha}\log n \leq z)=1.
\]
\end{theo}
\noindent
{\bf Remark:} One can consider the cover times along a subsequence 
$(n_k )_{k \geq 1}$ such that 
\[
\lim_{k \to \infty}\frac{N_{1/n_k}(A)}{-\log n_k}=\overline{\dim}_B(A).
\]
If in addition, 
\[
0<\liminf_{k \to \infty} n_k^{-\overline{\dim}_B(A)}N_{1/n_k}(A)
\leq \limsup_{k \to \infty} n_k^{-\overline{\dim}_B(A)}N_{1/n_k}(A)<\infty,
\]
then, one will obtain a result resembling Theorem \ref{thm:main} but along 
this subsequence, and with $\overline{\dim}_B(A)$ in place of $\dim_B(A)$.
We also see that Theorem \ref{thm:aux} covers the cases when 
\[
\liminf_{\delta \to 0} \delta^{\dim_B(A)}N_{\delta}(A)=0
\textrm{ or }
\limsup_{\delta \to 0} \delta^{\dim_B(A)}N_{\delta}(A)=\infty.
\]
That is, for those sets $A \subset \BR^d$ with a well defined box dimension 
but where assumption \eqref{eqn:dimliminfsup} fails.

\section{Preliminary results concerning the Poisson cylinder model}\label{sec:Cylprel}

The next result allows us to estimate the measure of the set of cylinders 
intersecting two distant balls. This lemma was first published as
Lemma 3.1 of \cite{TykessonWindisch}. Here, we present a sketch 
for sake of completeness. 

\begin{lemma}\label{lem:measurebounds_org}
Let $x_1,x_2\in \BR^d$ and let $r=d(x_1,x_2).$
There exists constants $c_1$ and $c_2$ depending only on $d$ such that
\[
\frac{c_1}{r^{d-1}} \leq \md{\L_{B(x_1, 1),B(x_2, 1)}} \leq 
\frac{c_2}{r^{d-1}} ,
\]
for every $r\geq 4.$
\end{lemma}

\noindent
{\bf Sketch of proof.}
By translation invariance of $\mu_{d,1}$, we can, without loss of generality 
assume that $x_1 = o$ so that $x_2$ is located on the surface of $B(o,r)$. 
Furthermore, we need order $r^{(d-1)}$ balls of radius 1 to cover the 
surface of $B(o,r)$. By symmetry, a random line passing through 
$B(o,1)$ will hit any fixed ball in the cover with equal probability. Thus, the
probability that it will hit $B(x_2,1)$ must be of order $r^{-(d-1)}$.
\fbox{}\\

\medskip

The next lemma states that, despite the long-range correlation nature of 
the Poisson cylinder process, events occurring in two distant sets are 
almost independent. The key point of the proof is to note that the events 
become independent after we conditioned on the event that no lines 
intersect both sets. 

\begin{lemma}\label{lem:almostindependecy}
 Let $K_1, K_2 \subset \bbR^d$ be disjoint sets and let $E_1$ and $E_2$ be events depending only on $\omega_t$ in $\L_{K_1}$ and $\L_{K_2}$ respectively. Then:
 $$\abs{\p{E_1 \cap E_2} - \p{E_1} \p{E_2}} \leq 4 \p{\omega_t(\L_{K_1,K_2}) \neq 0}.$$
\end{lemma}
\noindent
{\bf Proof.}
Note that
\begin{eqnarray}
 \lefteqn{\p{E_1 \cap E_2} } \label{eq:almostindependency1}\\
 & & =\pc{E_1}{\omega_t(\L_{K_1, K_2}) =0}\pc{E_2}{\omega_t(\L_{K_1, K_2}) =0}\p{\omega_t(\L_{K_1, K_2}) =0} \nonumber\\
 & & \ \ \ + \pc{E_1 \cap E_2}{\omega_t(\L_{K_1, K_2} \neq 0} \p{\omega_t(\L_{K_1, K_2}) \neq 0},\nonumber
\end{eqnarray}
since the events $E_1$ and $E_2$ are conditionally independent on $ \omega_t(\L_{K_1, K_2})=0$.
Furthermore, writing
\begin{eqnarray*}
\lefteqn{ \BP(E_i) 
=\pc{E_i}{\omega_t(\L_{K_1, K_2}) =0}\p{\omega_t(\L_{K_1, K_2}) =0}} \\
 && \ \ \ + \pc{E_i}{\omega_t(\L_{K_1, K_2}) \neq0}\p{\omega_t(\L_{K_1, K_2}) \neq0}
\end{eqnarray*}
 for $i= 1,2$ and using \eqref{eq:almostindependency1}, a straightforward calculation gives us that 
$$\abs{\p{E_1 \cap E_2} - \p{E_1} \p{E_2}} \leq 4 \p{\omega_t(\L_{K_1,K_2}) \neq 0},$$
as desired.
\fbox{}\\

As explained in the end of the introduction, one step in proving the 
main theorems is to show that the cover times of distant and 
small sets $K_1,\ldots,K_{n+1}$ are almost independent. If we 
consider one of these sets $K_j$, the next lemma gives us bounds 
on the probability that there exists 
a line passing through $K_j$ and any one of the other sets 
$K_i,$ $i \neq j.$

\begin{lemma}\label{lem:boundunion}
 Let $\{K_i\}_{i=1}^{n+1}$ be a family of sets such that 
 for every $i,$ $K_i \subset  B(x_i, 1)$ for some $x_i \in \bbR^d.$ 
Assume also that $\di{x_i}{x_j} = r_{ij} \geq 4 $ 
for every $i \neq j$ and let $r = \min_{i \neq j} r_{ij}$. 
 Then, for all $1 \leq j \leq n+1$ we have
 \[
 \p{\omega_t\left(\L_{\bigcup_{i=1} ^{n+1} K_i\setminus K_j, K_j}\right) \neq 0} 
 \leq  \frac{ntc_2}{r^{d-1}} ,
 \]
 where $c_2$ is the same as in Lemma \ref{lem:measurebounds_org}.
\end{lemma}
\noindent
{\bf Proof.}
Fix $j$ and let $B=\bigcup_{i=1} ^{n+1} B(x_i,1)$. By Lemma 
\ref{lem:measurebounds_org} and a simple union bound we have that
for any $i,$
\[
\md{\L_{\left(B \setminus B(x_i,1)\right),B(x_i,1)}} 
\leq nc_2\frac{1}{r^{d-1}}.
\]
Therefore we get that 
\begin{eqnarray*}
 \lefteqn{\p{\omega_t\left(\L_{\bigcup_{i=1} ^{n+1} K_i\setminus K_j, K_j}\right) \neq 0}
 \leq \p{\omega_t\left(\L_{B \setminus B(x_j,1), B(x_j,1)}\right) \neq 0}} \\
 &&  = 1 - \expo{-t \md{\L_{B \setminus B(x_j,1), B(x_j,1)}}}
 \leq 1 - \expo{-ntc_2\frac{1}{r^{d-1}}}.
\end{eqnarray*}
The result follows by using that $e^{-x}\geq 1-x$ for all $x\in\bbR$.
\fbox{}\\

Lemma \ref{lem:measurebounds_org} applies to 
$\mu_{d,1}(\L_{B(o,1)}\cap \L_{B(r e_1,1)})$
for large values of $r.$ Next, we need to consider the same expression but 
for small values
of $r.$ In contrast to Lemma \ref{lem:measurebounds_org}, no such result 
exists in the literature, and so we provide a full proof.

\begin{prop} \label{prop:smallbeta}
If $d\geq 3$ and $r\leq 2\sqrt{1-4^{-1/(d-2)}}$ then we have that
\begin{equation} \label{eqn:smallr}
\md{\L_{B(o,1)}\cap \L_{B(r e_1,1)}}
\leq 1-\frac{r}{12}.
\end{equation}
In addition, \eqref{eqn:smallr} holds for every $r\leq 2$ when $d=2.$
\end{prop}

\noindent
{\bf Proof.}
Observe that for any fixed $L\in G(d,1),$
the set of $y\in L^{\bot}$ such that $L+y\in \L_K$ is precisely 
$\Pi_{L^\bot}(K).$ From \eqref{eqn:repform1} it then follows that
\begin{eqnarray}\label{eqn:exactmeasure} 
\lefteqn{\md{\L_{B(o,1)}\cap \L_{B(r e_1,1)}}}\\
&&=\md{\L_{B(o,1)}}+\md{\L_{B(r e_1,1)}}-\md{\L_{B(o,1)\cup B(r e_1,1)}}
\nonumber\\
&&=\frac{1}{\kappa_{d-1}}\int_{G(d,1)} \int_{L^\perp}
\mathds{1}(L+y\in \L_{B(o,1)})+\mathds{1}(L+y\in \L_{B(r e_1,1)}) \nonumber \\
& & \hspace{40mm}-\mathds{1}(L+y\in \L_{B(o,1)\cup B(r e_1,1)})
\lambda_{d-1}(dy)
\nu_{d,1}(dL)\nonumber \\
&&=\frac{1}{\kappa_{d-1}}\int_{G(d,1)} \int_{L^\perp}
\mathds{1}(L+y\in \L_{B(o,1)})\mathds{1}(L+y\in \L_{B(r e_1,1)})
\lambda_{d-1}(dy)
\nu_{d,1}(dL)\nonumber \\
&&=\frac{1}{\kappa_{d-1}}\int_{G(d,1)}
\lambda_{d-1}\left(\Pi_{L^\bot}(B(o,1))\cap \Pi_{L^\bot}(B(r e_1,1))\right)
\nu_{d,1}(dL).\nonumber
\end{eqnarray}

If $L\in G(d,1)$ is written as
$L=\{s(l_1,...,l_d)\,:\,s\in {\mathbb R}\}$
where $l_1^2+...+l_d^2=1$, then the  projection matrix 
$\Pi_{L^{\bot}}$ has elements 
$(\Pi_{L^{\bot}})_{ii}=1-l_i^2$ and $(\Pi_{L^{\bot}})_{ij}=-l_i l_j$ for $i\neq j$.
Let $p_r=p_r(L)=\Pi_{L^\bot}((r,0,\ldots,0))$ so that $p_r$ is the projection
of the center of $B(r e_1,1).$ Of course, the projection of $B(o,1)$ onto $L^\bot$
is then a $(d-1)$-dimensional ball of radius 1 centred at $o.$
Straightforward calculations yield that  
\[
p_r=r(1-l_1^2,-l_1l_2,\ldots,-l_1l_d),
\]
so that
\[
|p_r|^2=r^2((1-l_1^2)^2+l_1^2l_2^2+\cdots+l_1^2l_d^2)
=r^2(1-2l_1^2+l_1^2(l_1^2+\cdots+l_d^2))=r^2(1-l_1^2).
\]
Of course, $\Pi_{L^\bot}(B(o,1))$ and $\Pi_{L^\bot}(B(r e_1,1))$ intersects whenever
$|p_r|\leq 2$, or equivalently whenever $ r^2(1-l_1^2)\leq 4.$ 
Since we are assuming
that $r\leq 2,$ this is always satisfied. 
Furthermore, the ($(d-1)$-dimensional) volume of the lens-shaped area
$\Pi_{L^\bot}(B(o,1))\cap \Pi_{L^\bot}(B(r e_1,1))$ is then
the sum of the volumes of two spherical caps of height $h=1-|p_r|/2.$ 
The volume of one such spherical
cap is (see \cite{Li}) given by
\[
\frac{1}{2}\kappa_{d-1}J_{2h-h^2}\left(\frac{d}{2},\frac{1}{2}\right).
\]
As above, $\kappa_{d-1}$ is the volume of the unit ball in $\bbR^{d-1}$
while $J_{2h -h^2}$ denotes a regularized
incomplete beta function. We note that 
$2h-h^2=1-|p_r|^2/4=1-r^2(1-l_1^2)/4,$
so that \eqref{eqn:exactmeasure} becomes
\begin{equation}\label{eqn:intsphere1}
\md{\L_{B(o,1)}\cap \L_{B(r e_1,1)}}
=\int_{G(d,1)}J_{1-r^2(1-l_1^2)/4}\left(\frac{d}{2},\frac{1}{2}\right)
\nu_{d,1}(dL).
\end{equation}
Furthermore, 
\[
J_{1-r^2(1-l_1^2)/4}\left(\frac{d}{2},\frac{1}{2}\right)
=\frac{\int_0^{1-r^2(1-l_1^2)/4}t^{\frac{d}{2}-1}(1-t)^{\frac{1}{2}-1}dt}
{\int_0^{1}t^{\frac{d}{2}-1}(1-t)^{\frac{1}{2}-1}dt}.
\]
Let $D_d:=\int_0^{1}t^{\frac{d}{2}-1}(1-t)^{-\frac{1}{2}}dt$ so that 
\begin{equation} \label{eqn:Jest1}
J_{1-r^2(1-l_1^2)/4}\left(\frac{d}{2},\frac{1}{2}\right)
=1-\frac{1}{D_{d}}\int_{1-r^2(1-l_1^2)/4}^1 
t^{\frac{d}{2}-1}(1-t)^{-\frac{1}{2}}dt.
\end{equation}
Furthermore, we trivially have that (since $d\geq 2,$)
\begin{equation} \label{eqn:Ddbound}
D_{d}=\int_0^1 t^{\frac{d}{2}-1}(1-t)^{-\frac{1}{2}}dt \leq 
\int_0^1 (1-t)^{-1/2}dt=2.
\end{equation}

We proceed to bound the integral on the right hand side of \eqref{eqn:Jest1}
from below. We have that 
\begin{eqnarray*}
\lefteqn{\int_{1-r^2(1-l_1^2)/4}^1 
t^{\frac{d}{2}-1}(1-t)^{-\frac{1}{2}}dt}\\
& & \geq (1-r^2(1-l_1^2)/4)^{\frac{d}{2}-1}\int_{1-r^2(1-l_1^2)/4}^1 
(1-t)^{-\frac{1}{2}}dt\\
& & =(1-r^2(1-l_1^2)/4)^{\frac{d}{2}-1} 2\sqrt{r^2(1-l_1^2)/4} \\
& & =(1-r^2(1-l_1^2)/4)^{\frac{d}{2}-1} r\sqrt{1-l_1^2},
\end{eqnarray*}
so that by \eqref{eqn:intsphere1}, \eqref{eqn:Jest1} and \eqref{eqn:Ddbound} 
we have that 
\begin{eqnarray} \label{eqn:mucapest}
\lefteqn{\mu_{d,1}\left(\L_{B(o,1)}\cap \L_{B(r e_1,1)}\right)}\\
& & \leq \int_{G(d,1)} 1-\frac{1}{2}
\left((1-r^2(1-l_1^2)/4)^{\frac{d}{2}-1} r\sqrt{1-l_1^2} \right) \nu_{d,1}(dL) 
\nonumber \\
& & =1-\frac{r}{2}\int_{G(d,1)}
(1-r^2(1-l_1^2)/4)^{\frac{d}{2}-1} \sqrt{1-l_1^2}  \nu_{d,1}(dL), \nonumber
\end{eqnarray}
which uses that $\nu_{d,1}(G(d,1))=1$ (see Section \ref{sec:models}).

In order to estimate the right hand side of \eqref{eqn:mucapest},
consider the sets $G_k:=\{L\in G(d,1): l_k^2>1/2\}.$ 
Since $l_1^2+\cdots+l_d^2=1,$ we clearly have that
$G_{k} \cap G_{m}=\emptyset$ if $k\neq m$. Therefore
\[
1=\nu_{d,1}(G(d,1))\geq \nu_{d,1}(\cup_{k=1}^d G_k)=d\nu_{d,1}(G_1),
\]
so that $\nu_{d,1}(G_1) \leq 1/d.$ We get that 
\begin{equation} \label{eqn:IntGestimate}
\int_{G(d,1)} \mathds{1}(|l_1|\leq 1/\sqrt{2})\nu_{d,1}(dL)
=1-\nu_{d,1}(G_1)\geq \frac{d-1}{d}.
\end{equation}
We will split \eqref{eqn:mucapest} into two cases. 
First, consider $d=2,$ so that 
\begin{eqnarray*}
\lefteqn{\mu_{2,1}({\L_{B(o,1)}\cap \L_{B(r e_1,1)}})}\\
& & \leq 1-\frac{r}{2}\int_{G(2,1)} \sqrt{1-l_1^2}\nu_{2,1}(dL)\\
& & \leq 1-\frac{r}{2}\int_{G(2,1)}\sqrt{1-l_1^2} \mathds{1}(|l_1|\leq 1/\sqrt{2})\nu_{2,1}(dL)\\
& & \leq 1-\frac{r}{2\sqrt{2}}
\int_{G(2,1)} \mathds{1}(|l_1|\leq 1/\sqrt{2})\nu_{2,1}(dL)
\leq 1-\frac{r}{4 \sqrt{2}},
\end{eqnarray*}
where we use \eqref{eqn:IntGestimate} in the last inequality.

Second, consider any $d\geq 3.$ Since $r\leq 2\sqrt{1-4^{-1/(d-2)}}$
it follows that 
\[
1-\frac{r^2}{4}\geq 2^{-2/(d-2)}.
\]
We therefore get that
\[
(1-r^2(1-l_1^2)/4)^{\frac{d}{2}-1}
\geq (1-r^2/4)^{\frac{d-2}{2}}\geq \frac{1}{2}.
\]
Hence, by \eqref{eqn:mucapest} and the above, 
we conclude that
\begin{eqnarray*}
\lefteqn{\md{\L_{B(o,1)}\cap \L_{B(r e_1,1)}}}\\
& & \leq 1-\frac{r}{2}\int_{G(d,1)}
(1-r^2(1-l_1^2)/4)^{\frac{d}{2}-1} \sqrt{1-l_1^2}
\nu_{d,1}(dL) \\
& & \leq 1-\frac{r}{4}\int_{G(d,1)} \sqrt{1-l_1^2}
\nu_{d,1}(dL)\\
& & \leq 1-\frac{r}{8}\int_{G(d,1)} \mathds{1}(|l_1|\leq 1/\sqrt{2})
\nu_{d,1}(dL)\\
& & \leq 1-\frac{r}{8}\frac{d-1}{d}
\leq 1-\frac{r}{12},
\end{eqnarray*}
where we use \eqref{eqn:IntGestimate} in the penultimate inequality.
\fbox{}\\

Define
\begin{equation}\label{eqn:betadef}
\beta(\rho,k) := \md{\L_{B(o,1)\cup B\left(2^k\rho e_1, 1\right)}}
\end{equation}
The next lemma gives both upper and lower bounds for $\beta(\rho,k)$, that 
will be used in the proofs of Lemma \ref{lem:boundsumofpairs}. 
The proof is an application of the previous proposition.
\begin{lemma}\label{lem:boundbeta}
 We have that $1 + \frac{2^k}{12}\rho< \beta(\rho,k) <2$, for 
 all $k$ such that 
 $2^k \rho \leq 2\sqrt{1-4^{-1/(d-2)}}$ for $d\geq 3$, or 
 $2^k \rho \leq 2$ when $d=2.$
\end{lemma}

\noindent
{\bf Proof.}
Clearly, $\beta(\rho,k) < 2\md{\L_{B(o,1)}} =2$. For the lower bound, note that, 
\begin{eqnarray*}
\lefteqn{\beta(\rho,k)
=\md{\L_{B(o,1)}\cup \L_{B\left(2^k \rho e_1,1\right)}}}\\
& & =\md{\L_{B(o,1)}}+\md{\L_{B\left(2^k \rho e_1,1\right)}}
-\md{\L_{B(o,1)}
\cap \L_{B\left( 2^k \rho e_1,1\right)}}\\
& & \geq 2-\left(1-2^k\frac{\rho}{12}\right)
= 1+\frac{2^k}{12}  \rho ,
\end{eqnarray*}
by using Proposition \ref{prop:smallbeta} with 
$r=2^k \rho.$ 
\fbox{}\\

\section{The packing lemma}\label{sec:packing}
As explained in the introduction, if $\tau$ is the 
cover time of $\BD$, then the set which is uncovered at time 
$\tau(1-\epsilon)$ consists (with high probability) of small and 
distant sets (this is Proposition \ref{prop:Gbound}). 
The proof of this is based on a second moment
argument, which involves a sum over pairs of points in $\BD$
(see Lemma \ref{lem:boundsumofpairs}).
In order to get a good bound for this sum, we will have to get 
estimates on the maximal number of points in $\BD$ within 
a certain distance of a fixed point $x\in \BD.$ In general, we need 
to obtain an estimate on the maximum number of $x\in \BD$ that belongs 
to a certain bounded subset of $\BR^d$. This is the main purpose of 
this section, and the result is presented in Lemma \ref{lem:packing} 
below, sometimes referred to as the packing lemma. However, in order 
to prove this lemma it is also convenient 
to introduce the concept of a maximal, $\rho$-separated subset of a bounded 
set $A.$ This maximal subset will play a crucial role in transferring 
results for finite sets (i.e. Theorem \ref{thm:fluct}) to the non-discrete
case (i.e. Theorem \ref{thm:main}).

Recall therefore the definition of a $\rho$-separated set. 
Given a bounded set $A \subset \bbR^d$ and $\rho \in (0,1),$ we let
$\BA^\rho \subset A$ be a $\rho$-separated set chosen so that the cardinality
$|\BA^\rho|$ is maximal among all possible choices of $\rho$-separated subsets of 
$A.$ There can be many candidates for such a set, and 
we therefore assume that $\BA^\rho$ is picked among these 
according to some predetermined rule. We define the ''discretization'' 
operator $\Delta^{\rho}$ which maps the set $A\subset \BR^d$ to the 
corresponding $\rho$-separated set $\BA^\rho$ so that $\Delta^\rho(A)=\BA^\rho.$ 
The introduction of $\Delta^\rho$ is only done in order to facilitate what would 
otherwise be overly cumbersome notation; it is more convenient to write 
$\Delta^\rho(A \cup B)$ or $\Delta^\rho(\alpha A)$ 
than any alternative. Of course, $\alpha>0$ and 
$\alpha A=\{\alpha x:x\in A\}$ for any set $A$. While we do not explicitly 
state which rule we use (as it will not be important), 
we stress that given $A \subset \BR^d,$ this rule determines $\BA^\rho$
without ambiguity. Furthermore, we shall assume that the rule we pick is 
such that for any $\alpha>0,$ 
\begin{equation} \label{eqn:scaleinv}
\alpha \BA^\rho=\alpha\Delta^\rho(A)=\Delta^{\alpha \rho}(\alpha A).
\end{equation}
The interpretation of \eqref{eqn:scaleinv} is simply that the discrete 
$\rho$-separated subset obtained from $A$ by $\Delta^\rho$, scaled by 
$\alpha>0$ (so that the separation is now $\alpha \rho$), is the same as 
if we started with the ($\alpha$-) scaled version of $A$, and {\em then} took the 
discrete $\alpha \rho$-separated subset.



Our first lemma establishes three results for $\BA^\rho$. This lemma will 
be used repeatedly in Section \ref{sec:mainproof}, but it will also 
be used in the proof of Lemma \ref{lem:packing}.
\begin{lemma}\mbox{}\label{lem:sizeofA}
\begin{itemize}

\item[a)]  For any $A\subset \BR^d$ bounded and all 
$\rho<\delta$ we have that $\abs{\BA^\delta}\leq \abs{\BA^\rho}$.

\item[b)]  Let $A_1\cap A_2=\emptyset$. 
Then we have that 
$|\Delta^\rho(A_1\cup A_2)|\leq |\BA_1^\rho|+|\BA_2^\rho|.$

\item[c)] For any set $A$ we have that for every $0<\rho<1,$
\[
|\BA^\rho|\leq 6^d \rho^{-d}|\BA^1|.
\]
\end{itemize}

\end{lemma}

\noindent
{\bf Proof.}
For part $a)$, simply observe that any $\delta$-separated set is also 
$\rho$-separated. Thus, $|\BA^\delta|\leq 
\max\{|\BD|:\BD \subset A, \Sep(\BD)\geq \rho\}=|\BA^\rho|$.

For part $b)$, assume that 
$|\Delta^\rho(A_1\cup A_2)|>|\BA_1^\rho|+|\BA_2^\rho|.$  
Then we get that 
\[
|\BA_1^\rho|+|\BA_2^\rho|
<|\Delta^\rho(A_1\cup A_2)\cap A_1|+|\Delta^\rho(A_1\cup A_2)\cap A_2|,
\]
and so without loss of generality we may assume that 
$|\BA_1^\rho|<|\Delta^\rho(A_1\cup A_2)\cap A_1|.$ 
However, $\Delta^\rho(A_1\cup A_2)\cap A_1$ is a 
$\rho$-separated set in $A_1,$ so this would contradict the maximality 
of $\BA_1^\rho.$

For part $c),$ we start by showing that 
$A \subset \bigcup_{x\in \BA^\rho} B(x,\rho)$. For contradiction, 
observe that if 
$A \not \subset \bigcup_{x\in \BA^\rho} B(x,\rho)$ 
we can find $y \in A \setminus \bigcup_{x\in \BA^\rho} B(x,\rho)$. 
Then, $y\in A$ and $\di{x}{y}\geq \rho$ for all $x\in \BA^\rho$. 
Therefore the set  $\{x_1, \dots, x_{|\BA^\rho|}, y\}\subset A$ is 
$\rho$-separated, contradicting the maximality of $|\BA^\rho|$.

Next, let $x_1, \dots x_{|\BA^1|}$ be an enumeration of the 
points in the set $\BA^1$. 
Set $D_1 = B(x_1,1)$ and then iteratively, let 
$D_j = B(x_j,1)\setminus \cup_{i=1}^{j-1}B(x_i,1).$ 
Obviously, $D_j \subset B(x_j,1)$, $D_i \cap D_j = \emptyset$ 
and $A \subset \cup_{i=1}^{|\BA^1|} B(x_i,1)=\cup_{i=1}^{|\BA^1|} D_i$. 
By using part $b)$ of this lemma we have that 
\begin{equation}\label{eqn:ABbound}
|\BA^\rho|\leq \left|\Delta^\rho \left(\bigcup_{i=1}^{|\BA^1|} D_i\right)\right|
\leq \sum_{i=1}^{|\BA^1|} |\BD_i^\rho| 
\leq \abs{\BA^1}|\BD_1^\rho|.
\end{equation}
Consistent with our previous notation, let 
$\BB(x,r)^\rho$ denote the maximal $\rho$-separated subset of the ball $B(x,r).$
We proceed to bound $|\BD_1^\rho|=\abs{\BB(o,1)^\rho}$.
To that end, note that
\[
\bigcup_{x \in \BB(o,1)^\rho}B(x,\rho/3) \subset B(o,1+\rho)
\]
and that $B(x_i,  \rho/3)\cap B(x_j, \rho/3) = \emptyset$ for all 
distinct $x_i,x_j \in \BB(o,1)^\rho$, since $\di{x_i}{x_j} \geq \rho$ 
by the definition of $\BB(o,1)^\rho$. Therefore (recall that $\lambda_d$ 
denotes $d$-dimensional Lebesgue measure),
\[
\abs{\BB(o,1)^\rho} \lambda_d(B(x_1, \rho/3)) \leq \lambda_d(B(o,1+\rho)) 
= \kappa_d (1+ \rho)^d \leq \kappa_d 2^d,
\] 
where we used $\rho < 1$ in the last inequality.
Since $\lambda_d(B(x_1,\rho/3)) = \kappa_d (\rho/3)^d$, it follows that 
$\abs{\BB(o,1)^\rho} \leq 6^{d} \rho^{-d}$. Inserting this into 
\eqref{eqn:ABbound} yields the result.
\fbox{}\\

We are now ready to state and prove our packing lemma.
Recall that $c$ represent a constant, depending only on the dimension $d$, 
and that it may change from line to line.

\medskip

\begin{lemma}\label{lem:packing}
There exists a constant $c < \infty $ depending on $d$ only such that
for any $\BD$ with $\Sep(\BD)\geq \rho$ we have that 
 \begin{enumerate}
  \item[a)] For any $\rho < r$, $\abs{\BD\cap B(o,r)} \leq c\rho^{-d}r^d$.
  \item[b)] For any $y \in \BD$ and $\rho<1$ we have that 
$\sum_{x\in \BD\setminus B(y,1)} \di{y}{x}^{1-d}
\leq c \rho^{-d}|\BD|^{1/d}. $
 \end{enumerate}
\end{lemma}

\noindent
{\bf Proof.}
\begin{enumerate}
\item[a)] Similar to the proof above, let 
$\{x_1,\dots, x_N\} = \BD \cap B(o,r)$ and observe that
$B(x_i, \rho) \subset B(o,r+\rho)$ for all $i = 1,\dots, N$. 
For any $j \neq i$ we have that 
$B(x_i,\rho/3)\cap B(x_j,\rho/3) = \emptyset$, since 
$\di{x_i}{x_j} \geq \rho.$ Therefore,
\[
 |\BD\cap B(o,r)| \lambda_d(B(x_1, \rho/3)) \leq  \lambda_d(B(o,r+\rho)) 
 = \kappa_d (r+ \rho)^d \leq c r^d,
\]
 where we used $\rho < r$ in the last inequality.
 Since $\lambda_d(B(x_1,\rho/3)) = \kappa_d (\rho/3)^d$ we have that
\[
|\BD\cap B(o,r)|\leq c \rho^{-d}r^{d}
\]
 and the proof is complete.

 \item[b)] Let 
$N(\rho,r, \BD)=\abs{\BD \cap \left(B(o,r+1)\setminus B(o,r)\right)}$. 
Our first step will be to estimate $N(\rho,r,\BD).$ Therefore,
let $\{x_1, \dots, x_{N(\rho,r,\BD)}\}$ denote the points in 
$\BD \cap \left(B(o,r+1)\setminus B(o,r)\right)$. For any 
$i = 1, \dots, N(\rho,r,\BD)$ we have that 
$B(x_i, \rho/3) \subset \left(B(o,r+2)\setminus B(o,r-1)\right)$. 
Also, for any $j \neq i$ we have that 
$B(x_i, \rho/3)\cap B(x_j, \rho/3) = \emptyset$, since 
$\di{x_i}{x_j} \geq \rho$ by the assumption on $\BD$. Thus,
 \begin{eqnarray*}
 \lefteqn{N(\rho,r,\BD) \lambda_d(B(x_1, \rho/3))}\\
 & & \leq \lambda_d(B(o,r+2)\setminus B(o,r-1)) 
 =\kappa_d \left((r+2)^d - (r-1)^d\right) \leq c r^{d-1}.
 \end{eqnarray*}
 Since $\lambda_d(B(x_1,\rho/3)) = \kappa_d (\rho/3)^d$ we conclude that
\begin{equation} \label{eqn:NDrhoest}
N(\rho,r,\BD) \leq c \rho^{-d}r^{d-1}.
\end{equation}

We assume without loss of generality that $y=o.$
Order the points $\{y_1,\dots,y_M\}=\BD \setminus\{o\}$ 
such that  $\di{o}{y_i} \leq \di{o}{y_j}$ for all $i <j$. Then, let 
 $N_k=|\BB(o,k)^\rho|$ and observe that for any $i>N_k$ we must have that
 $d(o,y_i)>k.$ Then, note that by Lemma \ref{lem:sizeofA} part $b)$, 
 and \eqref{eqn:NDrhoest}, we get that for $k\geq 2,$
 \begin{eqnarray} \label{eqn:Nkest}
 \lefteqn{N_k-N_{k-1}=|\BB(o,k)^\rho|-|\BB(o,k-1)^\rho| }\\
 & & \leq \left|\Delta^\rho \left(B(o,k)\setminus B(o,k-1)\right)\right| 
 \leq c \rho^{-d} k^{d-1}
 \leq c 2^{d-1}\rho^{-d}(k-1)^{d-1}. \nonumber
 \end{eqnarray}
We now define $K:=\max \{k: N_k < |\BD|\}$ and
consider the box $S_K=[-K/\sqrt{d},K/\sqrt{d}]^d\subset B(o,K)$. 
We then have that 
$\abs{\BS_K^\rho}\leq \abs{B(o,K)^\rho}=N_K$. Consider then 
the set $\rho \bbZ^d \cap S_K$ where 
$\rho \bbZ^d$ is the $d$-dimensional hypercubic lattice whose vertices 
are at distance $\rho$. Trivially, $\rho \bbZ^d \cap S_K$ is a 
$\rho$-separated set in $S_K$ and so we have that
$\abs{\BS_K^\rho} \geq \abs{\rho \bbZ^d \cap S_K} 
\geq 
c \rho^{-d}K^d$, by the maximality of 
$\abs{\BS_K^\rho}$. Therefore,
\[
c \rho^{-d}K^d \leq |\BS_K^\rho|\leq N_K < |\BD|,
\]
and therefore, $K \leq c \rho |\BD|^{1/d}.$ 

By part $a)$, $N_1 \leq c \rho^{-d}$ and so by letting $N_0=0$ and 
using Equation \eqref{eqn:Nkest} we get that,
\begin{eqnarray}
\lefteqn{\sum_{x\in \BD\setminus B(o,1)}\di{o}{x}^{1-d} 
\leq \sum_{i=1}^{|\BD|} \max(\di{o}{x_i},1)^{1-d}}\\
& &  
=\sum_{k=1}^{K+1} \sum_{i=N_{k-1}+1}^{\min(N_k,N)} \max(\di{o}{x_i},1)^{1-d} \nonumber \\
& & \leq N_1 +\sum_{k=2}^{K+1}(N_{k}-N_{k-1}) (k-1)^{1-d} \nonumber \\
& & \leq N_1 +\sum_{k=2}^{K+1}c \rho^{-d} (k-1)^{d-1} (k-1)^{1-d} \nonumber \\
& & \leq c \rho^{-d} +K c \rho^{-d}
\leq c \rho^{-d} +c \rho^{1-d}|\BD|^{1/d}
\leq c\rho^{-d} |\BD|^{1/d}\nonumber 
\end{eqnarray}
where we use that $\max(d(o,x_i),1)\geq 1$ for $1 \leq i \leq N_1$ 
and the observation
above that $d(o,x_i)\geq k-1$ for $N_{k-1}+1\leq i \leq N_k.$
\end{enumerate}
\fbox{}\\

\section{The fluctuation theorem and proof of Theorem \ref{thm:maindiscrete}}\label{sec:fluct}
The purpose of this section is to prove Theorem \ref{thm:maindiscrete}.
This will be done by proving Theorem \ref{thm:fluct}, which provides
a bound for the fluctuation of $\BP((\cT(\BD)-\log |\BD|\leq z)$ 
from the Gumbel distribution function. As mentioned in the introduction, 
the strategy is similar to that of \cite{Belius}.

In our first result we compare the cover time of a set with prescribed
separation, and the cover time of an ''independent'' (or infinitely 
separated if you will) set of the same cardinality.
The proof uses Lemmas \ref{lem:almostindependecy} and \ref{lem:boundunion}.

\begin{prop}\label{prop:almostindependentcover}
Let $m > 0$  and let $\BD$ be a finite set 
such that $\Sep(\BD)\geq \max((n/2)^\frac{m+2}{d-1},4)$.
Then, for $t\geq 0$ we have that
\[
\abs{\p{\cT\left(\BD\right)\leq t} 
- \p{\cT(o)\leq t}^n} \leq tc (n/2)^{-m}.
\]
\end{prop}

\noindent
{\bf Proof.}
Let $\BD=\{x_i\}_{i=1}^n,$ 
$\BD_i := \BD \setminus \{x_k\}_{k=1}^i$ 
and $B_i=\bigcup_{x\in \BD_i}B(x,1)$ . Then 
\begin{eqnarray*}
\lefteqn{\abs{\p{\cT(\BD)\leq t} 
- \p{\cT(o)\leq t}^n}}\\
 & &\leq \abs{\p{\cT(\BD) \leq t} - \p{\cT(\BD_1)\leq t}
 \p{\cT(x_1) \leq t}} \\
 & & \ \ + \p{\cT(o)\leq t}\abs{\p{\cT(\BD_1)\leq t} - \p{\cT(o) \leq t}^{n-1}},
\end{eqnarray*}
by translation invariance. 
Using Lemma \ref{lem:almostindependecy}, we get
\[
\abs{\p{\cT(\BD) \leq t} - 
\p{\cT(\BD_1)\leq t}\p{\cT(x_1) \leq t}}
\leq 4 \p{\omega_t\left(\L_{B_1 , B(x_1,1)}\right) \neq 0},
\]
and thus
\begin{eqnarray*}
\lefteqn{\abs{\p{\cT(\BD)\leq t} - \p{\cT(o)\leq t}^n}} \\
&&\leq 4\p{\omega_t\left(\L_{B_1 , B(x_1,1)}\right) \neq 0}  + \abs{\p{\cT(\BD_1)\leq t} - \p{\cT(o)\leq t}^{n-1}}.
\end{eqnarray*}
Repeating the same steps $n-1$ more times, we get:
\begin{eqnarray}
 \lefteqn{\abs{\p{\cT(\BD)\leq t} 
 - \p{\cT(o)\leq t}^n}} \nonumber\\
 && \leq 4\left( \p{\omega_t\left(\L_{B_1 , B(x_1,1)}\right)\neq 0}+\dots+ \p{\omega_t(\L_{B(x_{n},1),B(x_{n-1},1)})\neq 0}\right).\nonumber
\end{eqnarray}
We let $r:=\Sep(\BD)$ (we chose not to use $\rho,$ since here we think of 
the separation as being large rather than small).
Applying Lemma \ref{lem:boundunion} (which requires $r\geq 4$) to the 
above gives 
\begin{eqnarray*}
 \lefteqn{\abs{\p{\cT(\BD)\leq t} - \p{\cT(o)\leq t}^n}} \\
 &&\leq c \left((n-1) \frac{t}{r^{d-1}} +(n-2)\frac{ t}{r^{d-1}} + \dots + \frac{t}{r^{d-1}}\right)\\
 &&= c\frac{n(n-1)}{2}\frac{t}{r^{d-1}}.
\end{eqnarray*}
Therefore,
$$\abs{\p{\cT(\BD)\leq t} 
- \p{\cT(o)\leq t}^n} \leq n^2c \frac{ t}{r^{d-1}}.$$
Since $r \geq (n/2)^\frac{m+2}{d-1}$, we have that 
$r^{-(d-1)} \leq (n/2)^{-(m + 2)}$ and thus
\[
\abs{\p{\cT(\BD)\leq t} 
- \p{\cT(o)\leq t}^n}
\leq n^2 c\frac{t}{r^{d-1}} 
\leq t c (n/2)^{-m}.
\]
\fbox{}\\

Before we can proceed, we need to introduce some convenient notation.
For $0<\epsilon<1,$ let
\begin{equation} \label{eqn:Aepsdef}
\BD_\epsilon: = \{x \in \BD : \cT(x)
 >(1-\epsilon)\log |\BD|\}
\end{equation}
for any finite set $\BD$. 
The set $\BD_\epsilon$ corresponds to the subset of points 
$x \in \BD$ that are not covered at time $(1-\epsilon)\log\abs{\BD}$.  
Furthermore, we will let 
\begin{equation}\label{eqn:defCd}
C_d:=\sqrt{1-4^{-1/(d-2)}}
\end{equation}
where we interpret $C_2=1.$ Furthermore, let
\begin{equation}\label{eqn:defCdtilde}
\tilde{C}_d:=\frac{C_d}{12(1+C_d)}.
\end{equation}

Our next aim is Lemma \ref{lem:boundsumofpairs}, which will provide 
bounds on two sums 
(over pairs of points in $\BD$) of the probabilities of two distinct 
points being in the set $\BD_\epsilon$. 
This lemma is a preparation for the second moment argument of 
Proposition \ref{prop:Gbound} (mentioned in the introduction).
The strategy of the proof of Lemma \ref{lem:boundsumofpairs} is to 
split the sum into three parts.
The first part is a sum over all pairs of points that are within distance 
$C_d$ of each other, the second sums over pairs within the
intermediate region between $C_d$ and $\log |\BD|,$
while the last sums over distant points.
In order to obtain the desired result, we then use properties of 
$\mu_{d-1}$ along with the packing Lemma \ref{lem:packing}
and the estimate provided in Proposition \ref{prop:smallbeta}.

As we have mentioned before, our key results require that our sets
are separated ''enough''. To that end we shall have need of the following 
definition.
\begin{defi}\label{def:good}
Let $\BD$ be a finite set and let $\rho=\Sep(\BD).$ We say that $\BD$
is {\em good} if it satisfies the following conditions:
 $|\BD|^{\rho/800} \geq 2$, $|\BD|^{1/(2d)}\geq 4$ 
 and $\rho^{-d}(\log\abs{\BD})^d \leq \abs{\BD}^{\tilde{C}_d/2}.$ 
\end{defi}
\noindent
{\bf Remarks:} The definition of $\BD$ being good might look somewhat construed.
However, it will be convenient to have such a definition as we will appeal 
to it on many occasions below. It can be useful to
simply think of a good set $\BD$ as being ''large enough'' and 
''separated enough'', and the notion of goodness makes this precise. 

\medskip

The following lemma is stated without a proof as it is trivial to check.
One might argue that the condition in this lemma can be taken as the 
definition of good. However, it will be more convenient to have the three
explicit conditions of Definition \ref{def:good} at hand.

\begin{lemma} \label{lemma:goodset}
There exists a constant $D=D(d)<\infty$ such that any $\BD$ with 
\[
\Sep(\BD) \geq \log 2 \frac{800}{\log |\BD|}
\]
and $|\BD|\geq D$ is good.
\end{lemma}

\begin{lemma}\label{lem:boundsumofpairs}
Let $\BD$ be a good set. Then, for every  $\epsilon$ such that 
$\epsilon < \min\left(\frac{\rho}{36},\frac{5}{36},\frac{C_d}{148}\right)$
and $|\BD|^\epsilon \geq 2$
we have that
 \begin{enumerate}[label = \emph{\alph*})]
  \item $\sum\limits_{\substack{x,y\in \BD\\ x \neq y}} 
  \p{x,y \in \BD_\epsilon} < c\abs{\BD}^{-\epsilon} + \abs{\BD}^{2\epsilon}$
  \item $\sum\limits_{\substack{x,y \in \BD\\ 0 < \di{x}{y} < b \abs{\BD}^{1/2d}}} \p{x,y \in \BD_\epsilon} < cb^d\abs{\BD}^{-\epsilon},$
 \end{enumerate}
   for every $b\geq 1$, and where $c$ is a constant depending only on 
   the dimension $d$.
\end{lemma}
\noindent
{\bf Remark:} If it is the case that $\rho$ is so small that 
no such $\epsilon$ exists, then the statement is vacuous. 
A similar comment applies to other statements below. 


\medskip

\noindent
{\bf Proof.}
We will prove both statements by considering the sum 
 \[
 I= \sum\limits_{\substack{x,y \in \BD \\ 0 < \di{x}{y} \leq a}} \p{x,y \in \BD_\epsilon}
 \]
 and choosing appropriate values for $a$ later. 
Let us split $I$ into three parts:
 \begin{equation}\label{eq:I1}
 I_1 = \sum\limits_{\substack{x,y \in \BD \\ 
 \rho\leq \di{x}{y} < C_d}} \p{x,y \in \BD_\epsilon} ,
 \end{equation}
 \begin{equation}\label{eq:I2}
 I_2 = \sum\limits_{\substack{x,y \in \BD \\ C_d\leq \di{x}{y} <\log\abs{\BD}}} \p{x,y \in \BD_\epsilon},
 \end{equation}
and
\begin{equation}\label{eq:I3}
 I_3 = \sum\limits_{\substack{x,y \in \BD \\ \log\abs{\BD}\leq  \di{x}{y} \leq a}} \p{x,y \in \BD_\epsilon}. 
\end{equation}
Of course, it could be that $\rho\geq C_d$ or even that 
$\rho\geq \log |\BD|$. If so, $I_1$ and/or $I_2$ are simply zero while 
the analysis of $I_3$ remains the same. We will therefore assume that 
$\rho<C_d$ without any further comment.

In all three sums, we will use the following.
\begin{eqnarray}\label{eq:I1andI2}
\lefteqn{\p{x,y \in \BD_\epsilon}}\\
&&= \p{\{\cT(x) > (1-\epsilon) \log |\BD| \} 
\cap \{\cT(y) > (1-\epsilon) \log |\BD|\}}
\nonumber\\
&&= \p{\omega_{(1-\epsilon) \log |\BD|}
\left(\L_{B(x,1)}\cup \L_{B(y,1)}\right) = 0}\nonumber\\
&&= \expo{-(1-\epsilon) \log |\BD|
\md{\L_{B(x,1)} \cup \L_{B(y,1)}}}.\nonumber
\end{eqnarray}

\noindent
{\bf The sum $I_1$:}
Let $K = \lfloor \frac{\log C_d-\log \rho}{\log 2} \rfloor,$ 
where $\lfloor \cdot \rfloor$ denotes the integer part. Note that 
\[
2^{K+1} \rho \leq 2^{\frac{\log C_d-\log \rho}{\log 2} +1}\rho = 2C_d 
\textrm{ while } 
2^{K+1}\rho \geq 2^{\frac{\log C_d-\log \rho}{\log 2}}\rho = C_d.
\]
Note also that $K \geq 0$ since $\rho <C_d.$ 
Thus, \eqref{eq:I1} becomes
\begin{equation}\label{eqn:I1_2}
I_1 = \sum\limits_{\substack{x,y \in \BD \\ \rho\leq \di{x}{y} < C_d}} \p{x,y \in \BD_\epsilon}
\leq \sum_{k=0}^{K}\sum\limits_{\substack{x,y \in \BD \\ 2^k\rho\leq \di{x}{y} < 2^{k+1}\rho}} \p{x,y \in \BD_\epsilon}.
\end{equation}
For $2^k\rho \leq \di{x}{y} < 2^{k+1} \rho$, we have that
\begin{eqnarray*}
\lefteqn{\md{\L_{B(x,1)} \cup \L_{B(y,1)}} 
=\md{\L_{B(o,1)} \cup \L_{B(d(x,y)e_1,1)}} }\\
& & =\md{\L_{B(o,1)}}
+\md{\L_{B(d(x,y) e_1,1)}}
-\md{\L_{B(o,1)} \cap \L_{B(d(x,y) e_1,1)}}\\
& & \geq \md{\L_{B(o,1)}}
+\md{\L_{B(2^k\rho e_1,1)}}
-\md{\L_{B(o,1)} \cap \L_{B(2^k\rho e_1,1)}}\\
& & =\md{\L_{B(o,1)} \cup \L_{B(2^k \rho e_1,1)}}
\end{eqnarray*}
where the inequality follows since any line 
$L\in \L_{B(o,1)} \cap \L_{B(d(x,y)e_1,1)}$ must also 
belong to $\L_{B(o,1)} \cap \L_{B(2^k \rho e_1,1)}$.
Thus
\begin{eqnarray*}
\lefteqn{\expo{-(1-\epsilon) \log |\BD|
\md{\L_{B(x,1)} \cup \L_{B(y,1)}}} } \\
&&\leq \expo{-(1-\epsilon) \log |\BD|
\md{\L_{B(o,1)} \cup \L_{B(2^k \rho e_1,1)}}} \\
&&= \expo{-(1-\epsilon)\log\abs{\BD} \beta(\rho,k)}.
\end{eqnarray*}

Therefore, \eqref{eq:I1andI2}, \eqref{eqn:I1_2} and the above gives 
us that
\begin{eqnarray*}
\lefteqn{I_1 \leq \sum_{k=0}^{K}\sum\limits_{\substack{x,y \in \BD \\ 2^k\rho\leq  \di{x}{y} < 2^{k+1}\rho}} \expo{-(1-\epsilon)\log\abs{\BD} \beta(\rho,k)}}\\
&& \leq \sum_{k=0}^K \abs{\BD}^{1-(1-\epsilon)\left(1+\frac{2^k}{12}\rho\right)}c(2^{k+1}\rho)^d \rho^{-d} ,
\end{eqnarray*}
where we used Lemma \ref{lem:packing} part $a)$ and Lemma \ref{lem:boundbeta}
(which we can use since $2^K \rho \leq C_d$ as noted above). 
Furthermore, for any $k$  we claim that
\begin{equation}\label{eq:rhoepsi}
1 -(1-\epsilon)\left(1+\frac{2^k}{12}\rho\right) 
= -\frac{2^k}{12}\rho + \epsilon + \frac{2^k}{12}\epsilon \rho 
\leq - 2^k \epsilon.
\end{equation}
Indeed, \eqref{eq:rhoepsi} is equivalent to
\[
\epsilon \leq \rho \frac{\frac{1}{12}2^k}{1 + \frac{1}{12}2^k \rho + 2^k} 
= \rho \frac{\frac{1}{12}}{2^{-k} + \frac{1}{12} \rho + 1},
\]
and this holds since $\epsilon\leq \frac{\rho}{36}\leq \rho \frac{\frac{1}{12}}{2^{-k} + \frac{1}{12} \rho + 1}$ by our assumption on $\epsilon$.
Using \eqref{eq:rhoepsi} we get that
\begin{eqnarray*}
\lefteqn{I_1 \leq c \sum_{k=0}^K \abs{\BD}^{- 2^k \epsilon} 2^{d(k+1)}
\leq c\abs{\BD}^{-\epsilon} \sum_{k=0}^\infty \abs{\BD}^{- (2^k-1)\epsilon} 2^{d(k+1)}}\\
& & =c|\BD|^{-\epsilon}\left(2^d+|\BD|^{-\epsilon} 2^{2d}+\sum_{k=2}^\infty \abs{\BD}^{- (2^k-1)\epsilon} 2^{d(k+1)} \right) \\
& & \leq c|\BD|^{-\epsilon}\left(1+\sum_{k=1}^\infty |\BD|^{-2^k \epsilon} 2^{dk}
\right).
\end{eqnarray*}
Furthermore,
\begin{eqnarray*}
\lefteqn{\sum_{k=1}^\infty |\BD|^{-2^k \epsilon} 2^{dk}
\leq c\sum_{k=1}^\infty \int_{2^{k-1}}^{2^k} |\BD|^{-x \epsilon} x^{d-1} dx
\leq c\int_0^\infty e^{-x \epsilon \log |\BD|} x^{d-1} dx}\\
& & =c\left[x^{d-1} \frac{e^{-x \epsilon \log |\BD|} }{-\epsilon \log |\BD| }\right]_0^\infty
+c\frac{d-1}{\epsilon \log |\BD|}
\int_0^\infty e^{-x \epsilon \log |\BD|} x^{d-2} dx\\
& & =\cdots=\frac{c}{\left(\log |\BD|^{\epsilon}\right)^d},
\end{eqnarray*}
and by assumption $\log |\BD|^{\epsilon}\geq \log 2$ and so
\begin{equation}\label{eq:I1bounded}
I_1\leq  c |\BD|^{-\epsilon}.
\end{equation}

\noindent
{\bf The sum $I_2$:}
Consider now Equation \eqref{eq:I2}. Equation \eqref{eq:I1andI2} 
implies that
\begin{eqnarray*}
\lefteqn{I_2 = \sum\limits_{\substack{x,y \in \BD \\ 
C_d \leq  \di{x}{y} < \log\abs{\BD}}} \p{x,y \in \BD_\epsilon}
}\\
&& = \sum\limits_{\substack{x,y \in \BD 
\\ C_d \leq \di{x}{y} < \log\abs{\BD}}} 
\expo{-(1-\epsilon)\log\abs{\BD}\md{\L_{B(o,1)}\cup\L_{B(\di{x}{y} e_1,1)} }}\\
&& \leq \sum\limits_{\substack{x,y \in \BD \\ C_d \leq \di{x}{y} < \log\abs{\BD}}} \expo{-(1-\epsilon)\log\abs{\BD}\md{\L_{B(o,1)}\cup\L_{B\left(C_d e_1,1\right)} }},
\end{eqnarray*}
since $d(x,y)\geq C_d$. 
Now, Lemma \ref{lem:boundbeta} applied to $\rho=C_d$ and $k=0$, implies 
that $\md{\L_{B(o,1)}\cup\L_{B(C_d e_1,1)} } 
\geq 1+\frac{C_d}{12}$, 
which gives
 \begin{eqnarray*}
 \lefteqn{I_2 \leq \sum\limits_{\substack{x,y \in \BD \\ C_d\leq \di{x}{y} < \log\abs{\BD}}} \expo{-(1-\epsilon)\log\abs{\BD}\md{\L_{B(o,1)}
 \cup\L_{B\left(C_d e_1,1\right)} }}}\\
 &&\leq c\abs{\BD}\rho^{-d}(\log\abs{\BD})^d 
 \expo{-(1-\epsilon)\left(1+C_d/12\right)\log\abs{\BD}}\\
 && =c \rho^{-d}(\log\abs{\BD})^d 
 \abs{\BD}^{\epsilon- C_d/12+\epsilon C_d/12} 
 \leq c\abs{\BD}^{\epsilon-\tilde{C}_d/2+\epsilon\tilde{C}_d} 
 \end{eqnarray*}
where we used Lemma \ref{lem:packing} part $a)$ in the second 
inequality and the goodness of $\BD$ (in order to estimate 
$\rho^{-d}(\log|\BD|)^d$) in the last
inequality. Furthermore, we claim that 
$\epsilon-\tilde{C}_d/2+\epsilon\tilde{C}_d \leq -\epsilon$. 
Indeed, this follows since 
\[
\frac{\tilde{C}_d}{2(2+\tilde{C}_d)}
=\frac{C_d}{2(24(1+C_d)+C_d)}
\geq \frac{C_d}{148}\geq\epsilon,
\]
by our assumption on $\epsilon.$
We conclude that
\begin{equation}\label{eq:boundI2}
I_2 \leq c \abs{\BD}^{-\epsilon}.
\end{equation}

\noindent
{\bf The sum $I_3$:}
We now turn to Equation \eqref{eq:I3} and the sum $I_3$. 
Here we have that $d(x,y)\geq \log |\BD| \geq \log 4^{2d}\geq 4$
by the goodness of $\BD.$ Therefore, we can apply 
Lemma \ref{lem:measurebounds_org} to see that 
\begin{eqnarray*}
\lefteqn{\mu_{d,1}(\L_{B(x,1)} \cup \L_{B(y,1)})}\\
& & =2\mu_{d,1}(\L_{B(x,1)})-\mu_{d,1}(\L_{B(x,1)} \cap \L_{B(y,1)})
\geq 2 - \frac{c_2}{\di{x}{y}^{d-1}}.
\end{eqnarray*}
Then, we see from \eqref{eq:I3} that
\begin{eqnarray*}
\lefteqn{I_3 
\leq \sum\limits_{\substack{x,y \in \BD \\ \log\abs{\BD}\leq \di{x}{y} \leq a }} \expo{-(1 -\epsilon) \log\abs{\BD}\left(2 - \frac{c_2 }{\di{x}{y}^{d-1}}\right)}} \\
&& = \abs{\BD}^{-2(1-\epsilon)}\sum\limits_{\substack{x,y \in \BD \\ 
\log\abs{\BD}\leq \di{x}{y}\leq a }}
\exp\left((1-\epsilon)\log\abs{\BD}c_2\di{x}{y}^{1-d}\right).
\end{eqnarray*}
Since $\log\abs{\BD}\leq \di{x}{y}$, we have that 
\begin{eqnarray*}
(1-\epsilon)\log\abs{\BD}c_2\di{x}{y}^{1-d} 
 \leq (1-\epsilon)c_2(\log\abs{\BD})^{2-d}\leq c_2.
\end{eqnarray*}
It is easy to prove that for any $x \leq c_2$ we must have that 
$e^x \leq 1+e^{c_2}x$, and so we get that 
\begin{eqnarray} \label{eq:intermediateI2}
\lefteqn{\abs{\BD}^{-2(1-\epsilon)}
\sum\limits_{\substack{x,y \in \BD \\ \log\abs{A}\leq \di{x}{y}\leq a}}
\expo{(1-\epsilon)\log\abs{\BD}c_2\di{x}{y}^{1-d}}}  \\
&& \leq \abs{\BD}^{-2(1-\epsilon)}\sum\limits_{\substack{x,y \in \BD \\ 
\log\abs{\BD}\leq \di{x}{y} \leq a}} 
\left( 1 + e^{c_2}(1-\epsilon)\log\abs{\BD}\di{x}{y}^{1-d}\right) \nonumber \\
&& \leq \min\left( \abs{\BD}^{2\epsilon}, c\rho^{-d}\abs{\BD}^{2\epsilon - 1}a^d \right) \nonumber \\
&& \  + c(1-\epsilon)\abs{\BD}^{-2(1-\epsilon)}\log\abs{\BD}\sum\limits_{\substack{x,y \in \BD    \\ \log\abs{\BD}\leq \di{x}{y} \leq a}} 
\di{x}{y}^{1-d}  \nonumber
\end{eqnarray}
where the minimum comes from summing $1$ and using part $a)$
of Lemma \ref{lem:packing} to bound the number of elements in the sum.

Next, we provide a bound to the last term on the right hand side of 
\eqref{eq:intermediateI2}. To that end, note that
\[
\sum\limits_{\substack{x,y \in \BD \\ \log\abs{\BD}\leq \di{x}{y}\leq a}} \di{x}{y}^{1-d} 
\leq \sum_{x\in \BD}\sum_{y\in \BD\setminus B(x,1)} \di{x}{y}^{1-d}
 \leq \sum_{x\in \BD} c \rho^{-d}\abs{\BD}^{1/d}
 \leq c \rho^{-d}\abs{\BD}^{1+1/d},
\]
where we used part $b)$ of Lemma \ref{lem:packing} in the second inequality. 
Furthermore, this and the fact that $\rho^{-d}(\log\abs{\BD})^d 
\leq \abs{\BD}^{\tilde{C}_d/2}$ by the goodness of $\BD$ implies that 
\begin{eqnarray} \label{eqn:I3tmp}
\lefteqn{ c(1-\epsilon)\abs{\BD}^{-2(1-\epsilon)}\log\abs{\BD}
\sum\limits_{\substack{x,y \in \BD \\ \log\abs{\BD}\leq \di{x}{y} \leq a}} 
\di{x}{y}^{1-d}}\\
& & \leq c(1-\epsilon)\rho^{-d}\log\abs{\BD}\abs{\BD}^{-2(1-\epsilon)+1+1/d} 
\nonumber \\
& & \leq c\rho^{-d}(\log\abs{\BD})^d\abs{\BD}^{-2(1-\epsilon)+1+1/d}
\nonumber\\
& & \leq c\abs{\BD}^{-2(1-\epsilon)+1+1/d+\tilde{C}_d/2}
\leq c\abs{\BD}^{-\epsilon}. \nonumber
\end{eqnarray}
where we in the last inequality use that 
$-2(1-\epsilon)+1+1/d +\tilde{C}_d/2 \leq -\epsilon$.
Indeed, this follows since
\[
1-\frac{1}{d}-\frac{\tilde{C}_d}{2}
=1-\frac{1}{d}-\frac{C_d}{24(1+C_d)}
\geq 1-\frac{1}{d}-\frac{2}{24}\geq 
\frac{5}{12}
\geq 3 \epsilon,
\]
by our assumption on $\epsilon.$
Combining \eqref{eqn:I3tmp} and \eqref{eq:intermediateI2} we see that 
\begin{equation}\label{eq:I3bounded}
I_3 < \min\left(\abs{\BD}^{2\epsilon}, c\rho^{-d}\abs{\BD}^{2\epsilon -1}a^d\right)+ c\abs{\BD}^{-\epsilon}. \\
\end{equation}

By summing the contributions from $I_1,I_2$ and $I_3$
(\eqref{eq:I1bounded}, \eqref{eq:boundI2} and \eqref{eq:I3bounded}), we then conclude that
\[
 I=\sum\limits_{\substack{x,y \in \BD \\ 0 < \di{x}{y} \leq a}} \p{x,y \in \BD_\epsilon} 
 < \min\left(\abs{\BD}^{2\epsilon}, c\rho^{-d}\abs{\BD}^{2\epsilon -1}a^d\right)+ c\abs{\BD}^{-\epsilon}.
 \]
Taking the limit $a \to \infty$, we obtain 
\[
\sum\limits_{\substack{x,y\in \BD\\ x \neq y}} 
  \p{x,y \in \BD_\epsilon} < \abs{\BD}^{2\epsilon}+c\abs{\BD}^{-\epsilon},
\]
and the first statement is proved.

Furthermore, taking $a = b \abs{\BD}^{1/2d}$, we have that 
\begin{eqnarray*}
\lefteqn{c\rho^{-d}\abs{\BD}^{2\epsilon -1}a^d
=c\rho^{-d}\abs{\BD}^{2\epsilon -1/2}b^d
= c b^d |\BD|^{-\epsilon}
\left(\rho^{-d} |\BD|^{3\epsilon-1/2}\right)}\\
& & 
\leq c b^d |\BD|^{-\epsilon}\left( |\BD|^{3 \epsilon -1/2 + \tilde{C}_d/2}\right)
\leq c b^d |\BD|^{-\epsilon},
\end{eqnarray*}
where we used that 
$\rho^{-d}\leq \rho^{-d}(\log\abs{\BD})^d \leq \abs{\BD}^{\tilde{C}_d/2}$
by the goodness of $\BD$ and that 
$\abs{\BD}^{3\epsilon -1/2 + \tilde{C}_d/2} \leq 1$ 
(which as above follows by our assumption 
that $\epsilon< 5/36$).
Thus 
\[
\sum\limits_{\substack{x,y \in \BD\\ 0 < \di{x}{y} < b \abs{\BD}^{1/2d}}} \p{x,y \in \BD_\epsilon}  <c  b^d  \abs{\BD}^{-\epsilon}
\]
and the proof is complete.
\fbox{}\\

The next proposition is another crucial step towards the proof of Theorem
\ref{thm:maindiscrete}. Consider first 
 \begin{eqnarray}\label{eqn:defGArhoeps}
  \lefteqn{G_{\BD,\epsilon} = \{K \subset \BD: \abs{\abs{K} -\abs{\BD}^\epsilon} \leq \abs{\BD}^{2\epsilon/3} \mbox{ and }  } \\
  && \di{x}{y} \geq \abs{\BD}^{1/2d} \mbox{ for all distinct } x,y \in K \}, \nonumber
 \end{eqnarray}
so that $G_{\BD,\epsilon}$ is a collection of subsets of $\BD$ that 
are well separated and close in cardinality to $\abs{\BD}^\epsilon$. 
Also note that the condition $\abs{\abs{K} -\abs{\BD}^\epsilon} 
\leq \abs{\BD}^{2\epsilon/3} $ implies 
that $G_{\BD,\epsilon}$ consists only of non-empty sets.

\begin{prop}\label{prop:Gbound}
Let $\BD$ be a good set.  Then, for every  $\epsilon$ such that 
$\epsilon < \min\left(\frac{\rho}{36},\frac{5}{36},\frac{C_d}{148}\right)$
and $|\BD|^\epsilon \geq 2$
we have that
\[
\p{\BD_\epsilon \notin G_{\BD,\epsilon}} 
 \leq c\abs{\BD}^{-\epsilon/3}.
 \]
 Here, the constant $c$ only depends on $d.$
\end{prop}
\noindent
{\bf Proof.}
 We will get the result by proving that
\begin{equation}\label{eq:Gbound1}
\p{\exists x,y \in \BD_\epsilon: 0 < \di{x}{y} < \abs{\BD}^{1/2d}}
\leq c\abs{\BD}^{-\epsilon},
\end{equation} 
and
 \begin{equation}
  \p{\abs{\abs{\BD_\epsilon} - \abs{\BD}^\epsilon} 
  > \abs{\BD}^{2\epsilon/3}} \leq c\abs{\BD}^{-\epsilon/3}.\label{eq:Gbound2} 
 \end{equation}
 For \eqref{eq:Gbound1} we have:
 \begin{eqnarray*}
 \lefteqn{\p{\exists x,y \in \BD_\epsilon : 0 < \di{x}{y} 
 < \abs{\BD}^{1/2d}}}\\
& & =\p{\bigcup_{\substack{x,y \in \BD \\ 0 
< \di{x}{y} < \abs{\BD}^{1/2d}}} \{x,y \in \BD_\epsilon\}} \\
 && \leq \sum_{\substack{x,y \in \BD \\ 0 < \di{x}{y} < \abs{\BD}^{1/2d}}}  \p{x,y \in \BD_\epsilon} 
 \leq c\abs{\BD}^{-\epsilon},
 \end{eqnarray*}
 where the last inequality come from part $b)$ of 
 Lemma \ref{lem:boundsumofpairs} with $b = 1$, and \eqref{eq:Gbound1} is done.

For \eqref{eq:Gbound2}, we observe that since $\BP(\cT(x)>t)=e^{-t}$,
\[
\BE(|\BD_\epsilon|)
 =\sum_{x \in \BD} \BP(\cT(x)> (1-\epsilon)\log |\BD|)
 =|\BD| \exp\left(-(1-\epsilon)\log |\BD|\right)
 =|\BD|^{\epsilon}.
\]
By Chebyshev's inequality,
\begin{equation}
\p{\abs{\abs{\BD_\epsilon} - \abs{\BD}^\epsilon} > \abs{\BD}^{2\epsilon/3}} 
\leq \frac{\E{\abs{\BD_\epsilon}^2} - \abs{\BD}^{2\epsilon}}{\abs{\BD}^{4\epsilon/3}}. \label{eq:Chebyshev}
\end{equation}
Now
\begin{eqnarray*}
 \lefteqn{\E{\abs{\BD\epsilon}^2} 
 = \sum_{x,y \in \BD}\p{x,y \in \BD_\epsilon}} \\
 && = \sum_{x \in \BD}\p{x \in \BD_\epsilon} 
 + \sum\limits_{\substack{x,y \in \BD \\ x\neq y}}\p{x,y \in \BD_\epsilon} 
 \leq \abs{\BD}^\epsilon + c\abs{\BD}^{-\epsilon} + \abs{\BD}^{2\epsilon},
\end{eqnarray*}
by part $a)$ of Lemma \ref{lem:boundsumofpairs}.
Plugging this into \eqref{eq:Chebyshev} yields
\begin{eqnarray*}
 \lefteqn{\p{\abs{\abs{\BD_\epsilon} - \abs{\BD}^\epsilon} > \abs{\BD}^{2\epsilon/3}}}\\
 &&\leq \frac{\abs{\BD}^\epsilon + c\abs{\BD}^{-\epsilon} + \abs{\BD}^{2\epsilon} -\abs{\BD}^{2\epsilon}}{\abs{\BD}^{4\epsilon/3}}
 \leq \abs{\BD}^{-\epsilon/3} +c\abs{\BD}^{-7\epsilon/3} 
 \leq c\abs{\BD}^{-\epsilon/3},
\end{eqnarray*}
and the proof is complete. 
\fbox{}\\

We shall make use of the following inequality which we state
here (without argument) for convenience. 
\begin{equation}\label{eqn:calc1}
-x-x^2 \leq \log(1-x) \leq -x \textrm{ for every } x \in [0,1/2],
\end{equation}

We need a simple auxiliary result before we can present our fluctuation
theorem.
\begin{lemma}\label{lem:logexp}
Let $\BD$ be a finite set, and let $\epsilon>0$ be such that 
$|\BD|^\epsilon\geq 4.$ We then have that for any 
$z \geq -\frac{\epsilon}{4} \log|\BD|$
\begin{align} \label{eq:positive}
\abs{\left(1- \frac{e^{-z}}{\abs{\BD}^\epsilon}\right)^{\abs{\BD}^\epsilon +\abs{\BD}^{2\epsilon/3}} - \expo{-e^{-z}}} \leq 3\abs{\BD}^{-\epsilon/12},
\end{align}
and similarly,
\begin{align} \label{eq:negative}
\abs{\left(1- \frac{e^{-z}}{\abs{\BD}^\epsilon}\right)^{\abs{\BD}^\epsilon -\abs{\BD}^{2\epsilon/3}} - \expo{-e^{-z}}} \leq \abs{\BD}^{-\epsilon/12}.
\end{align}
\end{lemma}
\noindent
{\bf Remark:}
As the proof of this lemma is a straightforward, albeit somewhat
tedious, exercise we shall only provide a sketch. 

\medskip
{\bf Sketch of proof.} We only address \eqref{eq:positive}, as 
\eqref{eq:negative} is proved in essentially the same way.

Let $a=e^{-z}$ and $b=|\BD|^\epsilon$, and note that 
$a/b\leq |\BD|^{-3\epsilon/4} \leq  1/2$
by the assumption on $z$ and since $|\BD|^\epsilon\geq 4.$ 
Using this and 
\eqref{eqn:calc1}, it is then elementary to prove that 
\begin{equation} \label{eqn:abest1}
\left(1 - \frac{a}{b}\right)^{b + b^{2/3}}\leq \expo{-a}.
\end{equation}
One can then continue to show that 
\begin{eqnarray*}
\lefteqn{\abs{\expo{-a} -\left(1 - \frac{a}{b}\right)^{b + b^{2/3}}}}\\
& & \leq \expo{-a}\left(1  -\expo{- \left(\frac{a^2}{b} 
+ \frac{a}{b^{1/3}} +\frac{a^2}{b^{4/3}}\right)}\right)
\leq \left(1  -\expo{-3b^{-1/12}}\right),
\end{eqnarray*}
where we use \eqref{eqn:abest1}, the lower bound in \eqref{eqn:calc1}
and the definitions of $a,b.$ Finally, it is easy to see that  
 \begin{align*}
\left(1  -\expo{-3b^{-1/12}}\right) 
\leq 3 b^{-1/12} 
\leq 3\abs{\BD}^{-\epsilon/12}. 
 \end{align*}
Thus, \eqref{eq:positive} is proved.
\fbox{}\\

We are now ready to prove a theorem which gives an explicit bound
on the difference between the (centred) cover time $\cT(\BD)$ of a 
good set $\BD,$ and the distribution function of a Gumbel distribution.

\begin{theo}\label{thm:fluct} 
For any good set $\BD$ with $\rho=\Sep(\BD)$ and any $\epsilon$
such that $|\BD|^\epsilon \geq 2^{16}$ and 
$\epsilon < \min\left(\frac{\rho}{36},\frac{1}{12},\frac{C_d}{148}\right)$
we have that 
\[
\sup_{z\in \bbR}\abs{\p{\cT(\BD) -\log\abs{\BD} \leq z} -\exp(-e^{-z})}
\leq c_3\abs{\BD}^{-\epsilon/12},
\]
where $c_3$ is a constant depending on $d$ only. 
\end{theo}
\noindent
{\bf Remarks:} Note that the Theorem is only useful if 
$\abs{\BD}^{-\epsilon/12}$ is small or goes to 0 for a sequence 
$(\BD_n)_{n\geq 1}$. Thus, the assumption that 
$|\BD|^\epsilon\geq 2^{16}$ (which can be somewhat relaxed with additional work)
will turn out not to be a limitation.


\noindent
{\bf Proof.}
We are going to split the proof into three 
cases and start with the easy ones.

\paragraph{Case 1:} Consider $z \leq -\frac{\epsilon}{4}\log\abs{\BD}$. 
Then
\begin{eqnarray*}
 \lefteqn{\p{\cT(\BD) \leq(\log\abs{\BD} + z)}
\leq \p{\cT(\BD)\leq  \left(\log\abs{\BD} - \frac{\epsilon}{4}\log\abs{\BD}\right)}}\\
 &&= \p{\left\{x \in \BD : \cT(x) > \left(1 - \frac{\epsilon}{4}
 \right)\log\abs{\BD}\right\} = \emptyset} 
 = \p{\BD_{\epsilon/4} = \emptyset}.
\end{eqnarray*}
We have that $\p{\BD_{\epsilon/4} = \emptyset} 
\leq \p{\BD_{\epsilon/4}\notin G_{\BD,\epsilon/4}}$, 
since $G_{\BD,\epsilon/4}$ only contains non-empty sets. 
Thus, by Proposition \ref{prop:Gbound} we have that,
\begin{equation}
\p{\cT(\BD) \leq  (\log\abs{\BD} + z)} 
\leq \p{\BD_{\epsilon/4} \notin G_{\BD,\epsilon/4}} 
\leq c\abs{\BD}^{-\epsilon/12}.
\end{equation}
Using that $z\leq  -\log\abs{\BD}^{\epsilon/4}$ we have 
that $\exp(-e^{-z}) \leq \expo{-\abs{\BD}^{\epsilon/4}}
\leq \abs{\BD}^{- \epsilon/4}$.
Therefore, 
\begin{eqnarray}
\lefteqn{\abs{\p{\cT(\BD) \leq (\log\abs{\BD} + z)} -\exp(-e^{-z})}}\\
& & \leq \p{\cT(\BD) \leq (\log\abs{\BD} + z)} + \exp(-e^{-z})\nonumber \\
 && \leq c\abs{\BD}^{-\epsilon/12} + \abs{\BD}^{-\epsilon/4} < c\abs{\BD}^{-\epsilon/12}, \nonumber
\end{eqnarray}
and this ends the first case.
\paragraph{Case 2:} Assume $z \geq \epsilon\log\abs{\BD}$. This gives 
\begin{eqnarray*}
 \lefteqn{\p{\cT(\BD) > (\log\abs{\BD} + z)} 
 \leq  \p{\cT(\BD) > (1 +\epsilon)\log\abs{\BD}}}\\
 &&= \p{\bigcup_{x \in \BD} 
 \left\{\cT(x)>(1 +\epsilon)\log\abs{\BD}\right\} } 
  \leq \abs{\BD}\p{\cT(o) > (1 +\epsilon)\log\abs{\BD}}\\
 & & = \abs{\BD}\exp(-(1 +\epsilon)\log\abs{\BD}) = \abs{\BD}^{-\epsilon}.
\end{eqnarray*}
Using that $z \geq  \epsilon\log\abs{\BD}$ we have that 
$\exp(-e^{-z}) \geq \exp(-\abs{\BD}^{-\epsilon}) \geq 1 - \abs{\BD}^{-\epsilon}$ 
by the inequality $e^x \geq 1 + x$, 
which holds for all $x$. This, and the above equation gives
\begin{eqnarray} \label{eq:case2bound}
 \lefteqn{\abs{\p{\cT(\BD) \leq(\log\abs{\BD} + z)} -\exp(-e^{-z})}}\\
 &&=\abs{1-\p{\cT(\BD) > (\log\abs{\BD} + z)} -\exp(-e^{-z})} \nonumber\\
 &&\leq \p{\cT(\BD) > (\log\abs{\BD} + z)} + \abs{1 - \exp(- e^{-z})} 
 \leq \abs{\BD}^{-\epsilon} + \abs{\BD}^{-\epsilon} , \nonumber
\end{eqnarray}
which proves the case $z \geq\epsilon\log\abs{\BD}$.

 \paragraph{Case 3:} Assume that 
$z \in (-\frac{\epsilon}{4}\log\abs{\BD},\epsilon\log\abs{\BD})$ 
and start by observing that
\begin{eqnarray} \label{eq:firstterm}
 \lefteqn{\abs{\p{\cT(\BD)\leq \log\abs{\BD} + z} - \expo{-e^{-z}}}} \\
& & \leq \abs{\p{\cT(\BD) \leq \log\abs{\BD} + z} - \p{\cT(\BD) 
\leq \log\abs{\BD} + z, \BD_\epsilon \in G_{\BD,\epsilon}}} \nonumber \\
& &  \ \ \ +  \abs{\exp(-e^{-z})\p{\BD_\epsilon \in G_{\BD,\epsilon}}-\exp(-e^{-z})} \nonumber \\
& & \ \ \ + \abs{\p{\cT(\BD) \leq \log\abs{\BD} + z, 
\BD_\epsilon \in G_{\BD,\epsilon}} -\exp(-e^{-z})
\p{\BD_\epsilon \in G_{\BD,\epsilon}}} . \nonumber
\end{eqnarray}
We will now consider the three terms on the right hand side.

To deal with the first term, note simply that 
 \begin{eqnarray*}
  \lefteqn{\abs{\p{\cT(\BD) \leq \log\abs{\BD} + z} - \p{\cT(\BD) 
  \leq \log\abs{\BD} + z, \BD_\epsilon \in G_{\BD,\epsilon}}}} \\
  && = \p{\cT(\BD) \leq \log\abs{\BD} + z,\BD_\epsilon \notin G_{\BD,\epsilon}}
  \leq \p{\BD_\epsilon \notin G_{\BD,\epsilon}}
  \leq c\abs{\BD}^{-\epsilon/3},
 \end{eqnarray*}
 by Proposition \ref{prop:Gbound}.
  
For the second term of the right hand side in \eqref{eq:firstterm}, note that
\[
\abs{\expo{-e^{-z}}\p{\BD_\epsilon \in G_{\BD,\epsilon}}-\expo{-e^{-z}}}
=\expo{-e^{-z}} \p{\BD_\epsilon \notin G_{\BD,\epsilon}} 
\leq c\abs{\BD}^{-\epsilon/3}
\]
again by Proposition \ref{prop:Gbound}.

We now turn to the third term of the right hand side in \eqref{eq:firstterm},  
and this is where all our previous efforts come together.
We will show that for any $\BK \in G_{\BD,\epsilon}$,
\begin{equation} \label{eqn:Tsrhoprel}
\abs{\BP\left(\cT(\BD)
\leq \log |\BD|+z |\ \BD_\epsilon = \BK\right) 
-\expo{-e^{-z}}}
 \leq c\abs{\BD}^{-\epsilon/12}.
\end{equation}
Then, multiplying by $\p{\BD_\epsilon = \BK}$ and summing over all 
$\BK \in G_{\BD,\epsilon}$ on both sides, we get
\[
 \abs{\p{\cT(\BD) \leq \log |\BD|+z, 
\BD_\epsilon \in G_{\BD,\epsilon}} 
-\expo{-e^{-z}}\p{\BD_\epsilon \in G_{\BD,\epsilon}}}
\leq c\abs{\BD}^{-\epsilon/12}.
\]
We can then conclude from \eqref{eq:firstterm}, that 
\begin{equation}\label{eq:case3bound}
\abs{\p{\cT(\BD) \leq (\log\abs{\BD} + z)} 
-\expo{-e^{-z}}} \leq c\abs{\BD}^{-\epsilon/12}
\end{equation}
for all $z \in (-\epsilon/4 \log\abs{\BD}, \epsilon\log\abs{\BD})$,
and the proof will be complete.

In order to prove \eqref{eqn:Tsrhoprel}, define 
\[
\omega^1 = \{ L\in \omega_{(1-\epsilon)\log\abs{\BD}} : 
\exists x \in \BD \text{ such that } x \in \cyn{L}\},
\]
so that $\omega^1$ is the subset of $\omega_{(1-\epsilon)\log |\BD|}$ 
consisting of the cylinders actually covering a point $x \in \BD.$ 
Similarly, let 
\[
\omega^2 = \{ L\in \omega_{(1-\epsilon)\log\abs{\BD}, 
\log\abs{\BD} + z} 
: \exists x \in \BD \text{ such that } x \in \cyn{L}\}
\] 
so that a cylinder 
$\cyn{L}\in \omega^2$ arrives between times 
$(1 -\epsilon)\log |\BD|$ and $(\log |\BD|+z)$, and in addition 
covers some $x\in \BD.$ Finally, let $\omega^3=\omega^1 \cup \omega^2.$
For any $\omega\in \Omega,$ let $\cC(\omega,\BD)$ be the 
set of points in $\BD$ that are covered by the 
cylinders in $\omega$.

Fix $\BK \in G_{\BD,\epsilon}$ and define the event 
\[
E_1 := \{\BD_\epsilon= \BK\} 
= \{ \BD \setminus \BK =  \cC(\omega^1,\BD)\}.
\]
The equality holds since $\BK \subset \BD$ is the uncovered set
iff $\BD\setminus \BK$ is the covered set.
Furthermore, $\{\cT(\BD) \leq \log |\BD|+z\}$ is the event that all 
the points of $\BD$ are covered by time 
$(\log |\BD|+z).$ Hence, by the definition of $\omega^3,$
$\{\cT(\BD) \leq \log |\BD|+z\}
= \{\BD = \cC(\omega^3,\BD)\}$ so that
$E_1 \cap \{\cT(\BD) \leq \log |\BD|+z\} 
= E_1 \cap \{\BK \subset \cC(\omega^2,\BD)\}$. 
Letting $E_2 = \{\BK \subset \cC(\omega^2,\BD)\}$ and
using that $\omega^1$ and $\omega^2$
are independent, we have that $ \p{E_1 \cap E_2} = \p{E_1}\p{E_2}$.
Furthermore, since $\omega^2$ has the same distribution as a 
Poisson line process with intensity 
$\epsilon\log\abs{\BD} + z $, we get that
\[
\BP(E_2)=\p{\BK \subset \cC(\omega^2,\BD)}
= \p{\cT(\BK) \leq \epsilon\log\abs{\BD} + z}.
\]
On the other hand, 
\begin{eqnarray*}
\lefteqn{\p{E_2}=\frac{\p{E_1\cap E_2}}{\p{E_1}} }\\
& & =\frac{\p{E_1 \cap \left\{\cT(\BD)
\leq \log |\BD|+z\right\}}}{\p{E_1}}
= \BP\left(\cT(\BD)\leq \log |\BD|+z\Big| E_1\right),
\end{eqnarray*}
so we conclude that 
\begin{equation} \label{eq:firstbound} 
\BP\left(\cT(\BD)\leq \log |\BD|+z
\Big| \BD_\epsilon = \BK \right)
= \p{\cT(\BK) \leq \epsilon \log |\BD|+z}.
\end{equation}
Therefore, using \eqref{eq:firstbound} we have
\begin{eqnarray} \label{eqn:Tsrhoprel2}
\lefteqn{\abs{\BP(\cT(\BD)\leq \epsilon\log\abs{\BD} + z\ |\ 
\BD_\epsilon = \BK)  - \expo{-e^{-z}}}} \\
&&= \abs{\p{\cT(\BK) \leq \epsilon\log\abs{\BD} + z} -\expo{-e^{-z}}} \nonumber \\
&&\leq \abs{\p{\cT(\BK) \leq \epsilon\log\abs{\BD}+z} -\p{\cT(o) 
\leq \epsilon\log\abs{\BD}+z}^{\abs{\BK}}}
\nonumber \\
&& \ \ \ \ + \abs{\p{\cT(o) \leq \epsilon\log\abs{\BD}+z}^{\abs{\BK}} -\expo{-e^{-z}}}. \nonumber 
\end{eqnarray}
We will deal with the two terms on the right hand side of 
\eqref{eqn:Tsrhoprel2} separately. 

For the first term, let $x,y\in \BK$ be distinct. By the definition of 
$G_{\BD,\epsilon}$ we have 
$\di{x}{y}\geq \abs{\BD}^{1/(2d)} 
= (\abs{\BD}^\epsilon)^{1/(2\epsilon d)}$. 
Furthermore, since 
$\abs{\abs{\BK} - \abs{\BD}^\epsilon} \leq \abs{\BD}^{2\epsilon/3}$, 
we have that  
\begin{equation}\label{eq:intermediate1}
 \abs{\BD}^\epsilon - \abs{\BD}^{2\epsilon/3} 
 \leq \abs{\BK} \leq \abs{\BD}^\epsilon + \abs{\BD}^{2\epsilon/3}.
\end{equation}
In particular, \eqref{eq:intermediate1} implies $\abs{\BK} 
\leq 2\abs{\BD}^\epsilon $, so that $\di{x}{y} 
\geq (\abs{\BD}^\epsilon)^{1/(2\epsilon d)} 
\geq (\abs{\BK}/2)^{1/(2\epsilon d)}$. 
If we let $m = \frac{1}{2\epsilon}\frac{d-1}{d} -2$, we conclude that
$\di{x}{y} \geq (\abs{\BK}/2)^\frac{2+m}{d-1}$. Moreover, by the goodness
of $\BD$ we have that $d(x,y)\geq \abs{\BD}^{1/(2d)} \geq 4 $ and so we can 
use Proposition \ref{prop:almostindependentcover} with 
$n = \abs{\BK}$  together with the fact that 
$z \leq \epsilon\log\abs{\BD}$ to get
\begin{equation} \label{eqn:TsrhoK1}
\abs{\p{\cT(\BK)\leq \epsilon\log\abs{\BD}+z} 
-\p{\cT(o) \leq \epsilon\log\abs{\BD}+z}^{\abs{\BK}}}
\leq c\epsilon \log\abs{\BD}(\abs{\BK}/2)^{-m}.
\end{equation}
Furthermore, since $|\BD|^{\epsilon/3}\geq 2^{16/3}\geq 2$ 
by assumption, we see that
\[
\frac{2}{|\BK|}\leq \frac{2}{\abs{\BD}^\epsilon -\abs{\BD}^{2\epsilon/3}}
\leq \frac{1}{\abs{\BD}^{\epsilon/3}},
\]
where we use \eqref{eq:intermediate1} in the first inequality and the 
fact that $x^3-x^2 \geq 2x$ if $x\geq 2$ in the second.
Furthermore, since also $\epsilon < \frac{1}{12}\leq \frac{d-1}{6d}$ 
by assumption, we get that $m >1.$ We conclude that 
$(\abs{\BK}/2)^{-m}\leq \abs{\BD}^{-m\epsilon/3}\leq \abs{\BD}^{-\epsilon/3}$.
Thus,
\begin{equation} \label{eqn:TsrhoK2}
c\epsilon \log\abs{\BD}(\abs{\BK}/2)^{-m}
\leq c\epsilon(\log\abs{\BD})\abs{\BD}^{-\epsilon/3}.
\end{equation}
Using that $\log x  \leq x^{1/4}$ for any $x\geq 2^{16}$ and using the 
assumption that $|\BD|^\epsilon\geq 2^{16}$, we see that 
$\epsilon\log\abs{\BD} \leq \abs{\BD}^{\epsilon/4}$ so that 
$c\epsilon(\log\abs{\BD})\abs{\BD}^{-\epsilon/3} 
\leq c\abs{\BD}^{\epsilon/4}\abs{\BD}^{-\epsilon/3} 
= c\abs{\BD}^{-\epsilon/12}.$ 
We conclude from \eqref{eqn:TsrhoK1}
and \eqref{eqn:TsrhoK2} that 
\begin{equation} \label{eq:secondbound}
\abs{\p{\cT(\BK) \leq \epsilon\log\abs{\BD}+z } 
-\p{\cT(o) \leq \epsilon\log\abs{\BD}+z}^{\abs{K}}}
 \leq c\abs{\BD}^{-\epsilon/12}. 
\end{equation}

We can now turn to the second term of \eqref{eqn:Tsrhoprel2}.
We have that 
\[
\p{\cT(o) \leq \epsilon\log\abs{\BD}+z}^{\abs{\BK}} 
 = \left(1 - \exp(-\epsilon\log\abs{\BD} -z)\right)^{\abs{\BK}} = \left(1 - \frac{e^{-z}}{\abs{\BD}^\epsilon}\right)^{\abs{\BK}}.
\]
Then by \eqref{eq:intermediate1}, 
\begin{equation} \label{eq:thirdbound}
\left(1- \frac{e^{-z}}{\abs{\BD}^\epsilon}\right)^{\abs{\BD}^\epsilon +\abs{\BD}^{2\epsilon/3}}
\leq \p{\cT(o) \leq \epsilon\log\abs{\BD}+z}^{\abs{\BK}} 
\leq \left(1-\frac{e^{-z}}{\abs{\BD}^\epsilon}\right)^{\abs{\BD}^\epsilon - \abs{\BD}^{2\epsilon/3}}.
\end{equation}
Then, since $|\BD|^\epsilon\geq 2^{16}\geq 4$ we can use 
Lemma \ref{lem:logexp} to get that
(with the obvious meaning of $\pm$)
\begin{eqnarray*}
\lefteqn{\abs{\p{\cT(o) \leq \epsilon\log\abs{\BD}+z}^{\abs{\BK}} - \expo{-e^{-z}}}}\\
&& \leq \abs{\left(1- \frac{e^{-z}}{\abs{\BD}^\epsilon}\right)^{\abs{\BD}^\epsilon \pm\abs{\BD}^{2\epsilon/3}} - \expo{-e^{-z}}} \leq c\abs{\BD}^{-\epsilon/12}.
\end{eqnarray*}
Combining this with \eqref{eqn:Tsrhoprel2}, \eqref{eq:secondbound} and \eqref{eq:thirdbound}, 
the inequality \eqref{eqn:Tsrhoprel} is proved. This completes the proof.
\fbox{}\\

The proof of Theorem \ref{thm:maindiscrete} is now easy.

\noindent
{\bf Proof of Theorem \ref{thm:maindiscrete}}.
Firstly, for any set $\BD_n$ in our sequence, let $\rho_n=\Sep(\BD_n).$
Assume first that $\rho_n \to 0.$
According to Lemma \ref{lemma:goodset}, there exists some 
$N<\infty$ such that $\BD_n$ is good for every $n \geq N.$ 
Let 
\[
\epsilon_n=\frac{\rho_n}{50}
\]
and note that if $\BD_n$ is good, we have that 
\[
|\BD_n|^{\epsilon_n}=|\BD_n|^{\rho_n/50}\geq 2^{16}
\]
and so by perhaps increasing $N$ further, $\BD_n,\epsilon_n$ 
satisfies the assumptions of Theorem \ref{thm:fluct} 
for $n \geq N.$ Thus, for every $n\geq N,$ we have that 
\[
\sup_{z\in \bbR}\abs{\p{\cT(\BD_n) -\log\abs{\BD_n} \leq z} -\exp(-e^{-z})}
\leq c_3\abs{\BD_n}^{-\epsilon_n/12}=c_3\abs{\BD_n}^{-\rho_n/600}.
\]
The statement follows by letting $n \to \infty$ since
$\liminf_n \Sep(\BD_n) \log |\BD_n| =\infty$ by assumption.

If instead $\rho_n \not \to 0$ 
we can simply let 
$\epsilon_n=\min\left(\frac{\rho_n}{50},
\frac{1}{12},\frac{C_d}{148}\right)$ and again the statement follows 
in a similar way.
\fbox{}\\

\section{Proof of Theorem \ref{thm:main}} \label{sec:mainproof}

In order to prove Theorem \ref{thm:main}, we will study two quantities 
that are closely related to $\cT(A).$ Firstly, recall the definition of 
$\BA^\rho$ from Section \ref{sec:packing}. We will think of $\cT(A)$
and $\cT(\BA^\rho)$ as being generated by the same cylinder process, and
sometimes we will write $\cT(\BA^\rho,(\omega_t)_{t\geq 0})$ for emphasis.
We have the following easy 
proposition which we will not use, but we include for completeness
and motivation.
\begin{prop} \label{prop:TdTrealTw}
We have that for any $\rho>0,$
\[
\cT(\BA^\rho) \leq \limsup_{\delta \to 0} \cT(\BA^\delta) =\cT(A),
\]
where the inequality and equality holds for a.e. $(\omega_t)_{t\geq 0}.$
\end{prop}
{\bf Proof.} 
Obviously, $\cT(\BA^\rho,(\omega_t)_{t\geq 0})\leq \cT(A,(\omega_t)_{t\geq 0})$ 
for every $\rho>0.$ Assume that $\cT(A,(\omega)_{t\geq 0})>\tau$ and let 
\[
\cC_\tau:=\bigcup_{(L,s)\in \Psi:s \leq \tau} \cc(L).
\]
Then, there exists a point $x\in A\setminus \cC_\tau$
and some $\delta>0$ such that 
$B(x,\delta) \subset \BR^d \setminus \cC_\tau.$ This 
follows since a.s. $\cC_\tau$ is a closed set. Therefore we must 
have that $\cT(\BA^{\delta/2})>\tau.$ We conclude that 
$\limsup_{\delta \to 0} \cT(\BA^\delta)\geq \cT(A).$
\fbox{}\\

Of course, Proposition \ref{prop:TdTrealTw} tells us that in order to get 
a good approximation of $\cT(A)$ from below, one should estimate by 
$\cT(\BA^\rho)$ and take $\rho$ as small as possible. 
However, by taking $\rho$ too small, the estimates that we will obtain for 
$\cT(\BA^\rho)$ from Theorem \ref{thm:fluct} 
will become useless. Therefore, it will be important to pick $\rho$ in 
an optimal way. Later, we shall see that the ''sweet-spot'' is provided 
by letting $\rho$ be of order $(\log |\BA^1|)^{-1}.$

In order to define our second quantity, we start by saying that 
a ball $B(x,\rho)$ is 
{\em singularly covered} by time $t$ if 
there exists $L \in \omega_t$ such that $B(x,\rho) \subset \cyn{L}$. 
Alternatively, the point $x$ is $\rho$-\textit{singularly covered}
by time $t$ if the corresponding ball $B(x,\rho)$ is singularly covered.
Thus,
\[
\cT_s^\rho(x) 
:= \inf\{t >0 : \exists L \in \omega_t \mbox{ such that } B(x,\rho) \subset \cyn{L}\}
\]
is then the time at which the point $x$ is $\rho$-singularly covered.
Finally, we define the ($\rho$-) {\em well cover time} 
\begin{equation} \label{eqn:defTw}
\cT_w^\rho(A)
:=\inf\left\{t > 0: \cT_s^\rho(x)\leq t \
\forall x\in \BA^\rho \right\}.
\end{equation}
Again, $\cT_w^\rho(A)$ is generated by using the same cylinder process
used to generate $\cT(A)$ and $\cT(\BA^\rho)$.
It is an easy consequence of the above definitions that 
\begin{equation} \label{eqn:TdTTwcoupl}
\cT(\BA^\rho,(\omega_t)_{t \geq 0}) 
\leq \cT(A,(\omega_t)_{t \geq 0}) \leq \cT_w^\rho(A,(\omega_t)_{t \geq 0}),
\end{equation}
and these inequalities will provide the bridge
between the fluctuation result for discrete sets, i.e. Theorem \ref{thm:fluct}
and Theorem \ref{thm:main}. For the lower bound this is obvious, while for 
the upper we shall have use of the following proposition. 

\begin{prop} \label{prop:welldisc_scalingrel}
We have that 
\[
(1-\rho)^{d-1}\cT_w^\rho(A)
\stackrel{d}{=}
\cT(\Delta^{\rho/(1-\rho)}\left((1-\rho)^{-1}A\right)),
\]
where $\stackrel{d}{=}$ denotes equality in distribution.
\end{prop}
\noindent
{\bf Proof.}
Firstly, observe that the $\rho$-well cover time of $A$ equals the cover 
time of $\BA^\rho$, provided that the cylinders in our cylinder 
process were of radius $1-\rho$ instead of 1. Let 
$\cT^{1-\rho}(\BA^\rho)$ denote this cover time so that 
\[
\cT^{1-\rho}(\BA^\rho)=\cT_w^\rho(A).
\]

Secondly, we can scale space by a factor of $(1-\rho)^{-1}$ in order to 
re-obtain a cylinder process of radius 1. However, 
this scaling results in a cylinder process of rate $(1-\rho)^{d-1}$ 
(rather than 1). Indeed, start by recalling the notation from Section 
\ref{subsec:PCmodel} and the meaning of $\Delta^\rho$ from Section 
\ref{sec:packing}. As in the proof of Proposition \ref{prop:smallbeta},
observe that for any fixed $L\in G(d,1)$ the set of $y\in L^{\bot}$ 
such that $L+y\in \L_K$ is $\Pi_{L^\bot}(K).$
Furthermore, we see that 
\[
\int_{L^\bot} \mathds{1}(L+y\in \L_K)\lambda_{d-1}(dy)
=\lambda_{d-1}\left(\Pi_{L^{\bot}}(K)\right),
\]
for any compact $K \subset \BR^d.$
Thus, by equation \eqref{eqn:repform1},
\begin{eqnarray*}
\lefteqn{\md{\L_{K/(1-\rho)}}
=\frac{1}{\kappa_{d-1}}\int_{G(d,1)} 
\lambda_{d-1}\left(\Pi_{L^{\bot}}(K/(1-\rho))\right)
\nu_{d,1}(dL) }\\
& & =\frac{1}{\kappa_{d-1}(1-\rho)^{d-1}}\int_{G(d,1)} 
\lambda_{d-1}\left(\Pi_{L^{\bot}}(K)\right)
\nu_{d,1}(dL)=\frac{\md{\L_{K}}}{(1-\rho)^{d-1}},
\end{eqnarray*}
from which the scaling claim follows. Of course, after this scaling is 
performed, the set $\BA^\rho$ is now $\rho/(1-\rho)$-separated and
we obtain that (see \eqref{eqn:scaleinv})
\[
(1-\rho)^{-1}\BA^\rho=\Delta^{\rho/(1-\rho)}\left((1-\rho)^{-1}A\right).
\]
Therefore, the cover time $\cT^{1-\rho}(\BA^\rho)$ equals the cover time
of $\Delta^{\rho/(1-\rho)}\left((1-\rho)^{-1}A\right)$ when using a Poisson 
cylinder process with radius 1 and intensity $(1-\rho)^{d-1}.$ Thus,
\[
(1-\rho)^{d-1}\cT_w^\rho(A)=(1-\rho)^{d-1}\cT^{1-\rho}(\BA^\rho)
\stackrel{d}{=}
\cT(\Delta^{\rho/(1-\rho)}\left((1-\rho)^{-1}A\right)),
\]
as desired.
\fbox{}\\

\noindent
{\bf Remark:} Intuitively, when $\rho$ is small there will be little difference
between $\rho$ and $\rho/(1-\rho)$ and so 
\[
(1-\rho)^{-(d-1)}\cT(\Delta^{\rho/(1-\rho)}\left((1-\rho)^{-1}A\right))
\approx \cT(\BA^\rho).
\]
This means that we should be able to ''almost'' match the upper 
and lower bound in \eqref{eqn:TdTTwcoupl}. Of course, this is indeed 
the strategy of the proof of Theorem \ref{thm:main}.

\medskip

Before we proceed, we observe that it is not hard to prove that 
(see \cite{Falconer} p.30-31) for any $A$ and $\rho>0,$
\begin{equation} \label{eqn:NArhorel}
N_{2\rho}(A)\leq |\BA^\rho| \leq N_{\rho/2}(A).
\end{equation}
As a consequence, $N_\rho(A)$ can be replaced by $|\BA^\rho|$
in all the definitions of Section \ref{sec:dim} as well as in 
\eqref{eqn:dimliminfsup}.

We are now ready to prove the main theorem of the paper. \\

\noindent
{\bf Proof of Theorem \ref{thm:main}.}
We need to show that for any $\epsilon>0,$ there exists a constant $C_\epsilon$
independent of $n,$ such that for any set $A$ satisfying \eqref{eqn:dimliminfsup}
we have that 
\[
\BP(|\cT(nA)-\dim_B(A)(\log n+\log \log n)|\geq C_\epsilon)
\leq \epsilon,
\]
for every $n$.

Fix $\epsilon>0.$
For brevity, we define $A_n:=nA$ so that we can write $\BA_n^\rho$
in place of $\Delta^\rho(nA).$ 
Our first step is to let 
\[
\rho_n:=\frac{D}{\log |\BA_n^1|},
\]
where $800 \log 2 \leq D<\infty$, and observe that by Lemma \ref{lem:sizeofA}
part $a)$ we have that $\rho_n\geq 800 \log 2 (\log |\BA_n^{\rho_n}|)^{-1}$.
Then, according to Lemma \ref{lemma:goodset}, the set $\BA_n^{\rho_n}$ 
is good whenever $|\BA_n^1|$ (and therefore also $|\BA_n^{\rho_n}|$)
is large enough. From now one we will simply assume that $n$ is so large that 
$\BA_n^{\rho_n}$ is a good set. For reasons that will transpire, we shall
further assume that $D$ is such that $c_3 e^{-D/600}\leq \epsilon/4$,
and (again for reasons that will become clear) 
we choose $C_\epsilon\geq 12d D$ such that $e^{-C_\epsilon/8}\leq \epsilon/4$
and $\exp(-e^{C_\epsilon/8})\leq \epsilon/4.$ 

Assume first that $A$ satisfies the assumption that  
for some $0<c_A<\infty,$
\begin{equation} \label{eqn:dimlim}
\lim_{\rho \to 0} \rho^{\dim_B(A)} |\BA^\rho|=c_A.
\end{equation}
This assumption is clearly stronger than \eqref{eqn:dimliminfsup}, and the 
reason for this stronger assumption is to illustrate how the constant 
$c_A$ comes in to play (see also the remark after this proof).
We will also let $\tilde{c}_A:=\dim_B(A)^{\dim_B(A)} c_A$.

It follows from \eqref{eqn:scaleinv} that 
for every $n\geq 1$, $\abs{\BA_n^{\rho_n}}=\abs{n\BA^{\rho_n/n}}
=\abs{\BA^{\rho_n/n}}$ since enlarging a set does not change its cardinality. 
Therefore,
\[
\lim_{n \to \infty}\frac{\log |\BA_n^1|}{\log n}
=\lim_{n \to \infty}\frac{\log |\BA^{1/n}|}{\log n}
=\dim_B(A),
\]
by the definition of $\dim_B(A)$ (see Section \ref{sec:dim}).
We then get that 
\begin{eqnarray*} 
\lefteqn{\left(\frac{D}{n\log n}\right)^{\dim_B(A)}|\BA_n^{\rho_n}|
=\left(\frac{\log |\BA_n^1|}{\log n}\right)^{\dim_B(A)}
\left(\frac{D}{n\log |\BA_n^1|}\right)^{\dim_B(A)}|\BA_n^{\rho_n}| }\\
& & =\left(\frac{\log |\BA_n^1|}{\log n}\right)^{\dim_B(A)}
\left(\frac{\rho_n}{n}\right)^{\dim_B(A)}|\BA^{\rho_n/n}|
\to \dim_B(A)^{\dim_B(A)} c_A=\tilde{c}_A, 
\end{eqnarray*}
by \eqref{eqn:dimlim} since $\rho_n/n \to 0.$ Thus, 
\begin{equation}\label{eqn:anrhonlim}
\frac{1}{\tilde{c}_A}
\left(\frac{D}{n\log n}\right)^{\dim_B(A)}|\BA_n^{\rho_n}| \to 1.
\end{equation}
Note further that $|\log D^{\dim_B(A)}|\leq |d\log D| 
\leq d D\leq C_\epsilon/12$ by one of our assumptions on $C_\epsilon.$
Thus, by \eqref{eqn:anrhonlim} we have that 
\[
|\log D^{\dim_B(A)}|+
\left|\log \left(\frac{1}{\tilde{c}_A}
\left(\frac{D}{n\log n}\right)^{\dim_B(A)}|\BA_n^{\rho_n}|
\right)\right| \leq C_\epsilon/2,
\]
for every $n$ larger than some $N(D,C_\epsilon)$. 
Thus, for $n \geq N(D,C_\epsilon)$ we get
\begin{eqnarray}\label{eqn:Trealest1}
\lefteqn{\BP(|\cT(A_n)
-\dim_B(A)(\log n+\log \log n)-\log \tilde{c}_A )| \geq C_\epsilon)}\\
& & =\BP\Big(\Big|\cT(A_n)-\log |\BA_n^{\rho_n}|-\log D^{\dim_B(A)}
\nonumber\\
& & \ \ \ \ +\log \left(\frac{1}{\tilde{c}_A}
\left(\frac{D}{n\log n}\right)^{\dim_B(A)}|\BA_n^{\rho_n}|
\right)\Bigg|\geq C_\epsilon\Bigg)\nonumber\\
& & \leq \BP\left(\left|\cT(A_n)-\log |\BA_n^{\rho_n}|\right|
\geq C_\epsilon/2\right) \nonumber\\
& & \leq \BP\left(\cT(\BA_n^{\rho_n})-\log |\BA_n^{\rho_n}| \leq 
-C_\epsilon/2\right)
+\BP\left(\cT_w^{\rho_n}(A_n)-\log |\BA_n^{\rho_n}|
\geq C_\epsilon/2\right), \nonumber
\end{eqnarray}
where the last inequality follows since 
$\cT(\BA_n^{\rho_n}) \leq \cT(A_n) \leq \cT_w^{\rho_n}(A_n)$.

We shall now address the two probabilities of the right hand side of 
\eqref{eqn:Trealest1} separately.
First, we can apply Theorem \ref{thm:fluct} with $\epsilon_n=\rho_n/50$ 
to conclude that
\[
|\BP\left(\cT(\BA_n^{\rho_n})-\log |\BA_n^{\rho_n}|
\leq -C_\epsilon/2\right)-\exp(-e^{C_\epsilon/2})| 
\leq c_3 |\BA_n^{\rho_n}|^{-\rho_n/600}
\]
for every $n\geq N(D,C_\epsilon)$ (by perhaps making $N(D,C_\epsilon)$
even larger than before). Then, by using Lemma \ref{lem:sizeofA}
part $a)$ we see that 
\begin{eqnarray} \label{eqn:Tdest1}
\lefteqn{\BP\left(\cT(\BA_n^{\rho_n})-\log |\BA_n^{\rho_n}|
\leq -C_\epsilon/2\right)
\leq \exp(-e^{C_\epsilon/2})+c_3 |\BA_n^{\rho_n}|^{-\rho_n/600}}\\
& & \leq \exp(-e^{C_\epsilon/2})+c_3 |\BA_n^1|^{-\rho_n/600}
=\exp(-e^{C_\epsilon/2})+c_3 e^{-D/600}\leq \epsilon/2, \nonumber
\end{eqnarray}
because of our choices of $D$ and $C_\epsilon.$

We now turn to the second term of the right hand side of 
\eqref{eqn:Trealest1}. By Proposition \ref{prop:welldisc_scalingrel} we
have that $(1-\rho_n)^{d-1}\cT_w^{\rho_n}(A_n)=
\cT(\Delta^{\rho_n/(1-\rho_n)}((1-\rho_n)^{-1}A_n))$.
Furthermore, as above it follows from \eqref{eqn:scaleinv} that 
$|\Delta^{\rho_n/(1-\rho_n)}((1-\rho_n)^{-1}A_n)|
=|(1-\rho_n)^{-1}\BA_n^{\rho_n}|=|\BA_n^{\rho_n}|.$ 
We can therefore use Theorem \ref{thm:fluct} to conclude that 
\begin{equation} \label{eqn:wellfluct}
\left|\BP\left((1-\rho_n)^{d-1}\cT_w^{\rho_n}(A_n)-\log |\BA_n^{\rho_n}|
\leq z \right)
-\exp\left(-e^{-z}\right)\right| \leq c_3  |\BA_n^{\rho_n}|^{-\rho_n/600}
\end{equation}
for every $z\in \BR^d.$
Indeed, since $\rho_n/(1-\rho_n)>\rho_n \geq D/\log |\BA_n^{\rho_n}|,$ the
set $\Delta^{\rho_n/(1-\rho_n)}((1-\rho_n)^{-1}A_n)$ is good and Theorem 
\ref{thm:fluct} can be applied for every $n\geq N(D,C_\epsilon).$

By yet again picking $N(D,C_\epsilon)$ perhaps even larger than before, 
we have that 
$\rho_n=D/\log|\BA_n^1|\leq 1-2^{-1/d}$ for $n \geq N(D,C_\epsilon)$. 
We can then use
the inequality $(1-x)^{-d+1}\leq 1+2dx$ which holds for 
$0<x\leq 1-2^{-1/d}$ to conclude that for such $n,$
$(1-\rho_n)^{-(d-1)}\leq 2d \rho_n.$
Therefore, by Lemma \ref{lem:sizeofA} part $c)$ we get that 
\begin{eqnarray*}
\lefteqn{\log |\BA_n^{\rho_n}|\left((1-\rho_n)^{-(d-1)}-1\right)
\leq \log (6^d \rho_n^{-d} |\BA_n^1|)2d\rho_n}\\
& &= (d\log 6+\log |\BA_n^1|+d\log \log |\BA_n^1|-d\log \log D)
2d \frac{D}{\log |\BA_n^1|}.
\end{eqnarray*}
Clearly, there exists an $N(D,C_\epsilon)$ perhaps even larger than before, 
such that for every $n\geq N(D,C_\epsilon),$
\[
\log |\BA_n^{\rho_n}|\left((1-\rho_n)^{-(d-1)}-1\right)\leq 3dD
\leq C_\epsilon/4,
\]
where we use the fact that $C_\epsilon \geq 12dD$ by assumption.

Hence, for $n\geq N(D,C_\epsilon),$ 
\begin{eqnarray*} 
\lefteqn{\BP\left(T_w^{\rho_n}(A_n)-\log |\BA_n^{\rho_n}|
\geq C_\epsilon/2\right)}\\
& & = \BP\left(T_w^{\rho_n}(A_n)
-\frac{1}{(1-\rho_n)^{d-1}}\log |\BA_n^{\rho_n}|
+\frac{1}{(1-\rho_n)^{d-1}}\log |\BA_n^{\rho_n}|
-\log |\BA_n^{\rho_n}|\geq C_\epsilon/2\right)\\
&& \leq \BP\left(T_w^{\rho_n}(A_n)
-\frac{1}{(1-\rho_n)^{d-1}}\log |\BA_n^{\rho_n}|
\geq C_\epsilon/4\right) \\
& & = \BP\left((1-\rho_n)^{d-1}T_w^{\rho_n}(A_n)-\log |\BA_n^{\rho_n}|
\geq (1-\rho_n)^{d-1}C_\epsilon/4\right)\\
& & \leq \BP\left((1-\rho_n)^{d-1}T_w^{\rho_n}(A_n)-\log |\BA_n^{\rho_n}|
\geq C_\epsilon/8\right) 
\end{eqnarray*}
where the last inequality holds for every $n$ such that 
$(1-\rho_n)^{d-1} \geq 1/2.$ This clearly holds for every 
$n \geq N(D,C_\epsilon)$ where $N(D,C_\epsilon)$ 
might be even larger than before.

We can now use \eqref{eqn:wellfluct} to see that 
\begin{eqnarray} \label{eqn:Twest1}
\lefteqn{\BP\left(T_w^{\rho_n}(A_n)-\log |\BA_n^{\rho_n}|
\geq C_\epsilon/2\right)}\\ \nonumber
& & \leq \BP\left((1-\rho_n)^{d-1}T_w^{\rho_n}(A_n)-\log |\BA_n^{\rho_n}|
\geq C_\epsilon/8\right) \\ \nonumber
&& =1-\BP\left((1-\rho_n)^{d-1}T_w^{\rho_n}(A_n)-\log |\BA_n^{\rho_n}|
\leq C_\epsilon/8\right) \\ \nonumber
& & \leq 1-\exp(-e^{-C_\epsilon/8})+c_3|\BA_n^{\rho_n}|^{-\rho_n/600}
\leq e^{-C_\epsilon/8}+c_3 e^{-D/600}\leq \epsilon/2, \nonumber
\end{eqnarray} 
much as when we dealt with the first term of the right hand side of 
\eqref{eqn:Trealest1}.

Finally, combining \eqref{eqn:Trealest1}, \eqref{eqn:Tdest1} 
and \eqref{eqn:Twest1} 
we conclude that 
\begin{equation} \label{eqn:finalineq}
\BP(|\cT(A_n)-\dim_B(A)(\log n+\log \log n)-\log \tilde{c}_A )|
\geq C_\epsilon) \leq \epsilon
\end{equation}
for every $n \geq N(D,C_\epsilon)$. However it is now easy to see 
that \eqref{eqn:finalineq} must in fact hold for every $n\geq 1$ by 
(perhaps) increasing $C_\epsilon$ even further.

The full statement (i.e. assuming \eqref{eqn:dimliminfsup} instead of 
\eqref{eqn:dimlim}) is proved in a very similar way, and therefore
we will only indicate the changes. Again we get that 
$\lim_{n \to \infty}\frac{\log |\BA_n^1|}{\log n} =\dim_B(A)$ and
the first change is that
\eqref{eqn:anrhonlim} is replaced by 
\[
0<\liminf_{n \to \infty} \left(\frac{D}{n\log n}\right)^{\dim_B(A)}|\BA_n^{\rho_n}| 
\leq \limsup_{n \to \infty} \left(\frac{D}{n\log n}\right)^{\dim_B(A)}|\BA_n^{\rho_n}|
<\infty.
\]
Then, one can pick $C_\epsilon$ perhaps even larger so that  
\[
|\log D^{\dim_B(A)} |
+\limsup_{n \to \infty} \left|\log\left(\left(\frac{D}{n\log n}\right)^{\dim_B(A)}|\BA_n^{\rho_n}|\right)\right| \leq C_\epsilon/3.
\]
This can then be inserted into a slightly modified version of 
\eqref{eqn:Trealest1} in order to obtain the statement
\begin{eqnarray*}
\lefteqn{\BP(|\cT(A_n)
-\dim_B(A)(\log n+\log \log n)| \geq C_\epsilon)}\\
& & \leq \BP\left(\cT(\BA_n^{\rho_n})-\log |\BA_n^{\rho_n}|
\leq -C_\epsilon/2\right) \\
& & \ \ \ \ +\BP\left((1-\rho_n)^{d-1}\cT_w^{\rho_n}(A_n)-\log |\BA_n^{\rho_n}|
\geq C_\epsilon/2\right),
\end{eqnarray*}
for every $n$ large enough, and then we can proceed as above.
\fbox{}\\

\noindent
{\bf Remark:} Recall the remark after the statement of Theorem \ref{thm:main}
where it is speculated that under strong enough regularity conditions
$(\cT(nA)-\dim_B(A)(\log n+\log \log n) +C)_{n \geq 1}$ might converge 
to a Gumbel distribution for some constant $C.$ Then, 
\eqref{eqn:finalineq} indicates that 
the constant $C$ might be the same as 
$-\log \tilde{c}_A=-\dim_B(A)\log \dim_B(A)-\log c_A$. This is our reason 
for first proving Theorem \ref{thm:main} under the stronger assumption 
\eqref{eqn:dimlim}.

\section{Proof of Theorem \ref{thm:aux}} \label{sec:secondaryproof}
In this section we shall prove our secondary result, i.e. Theorem 
\ref{thm:aux}. 

\noindent
{\bf Proof of Theorem \ref{thm:aux}.}
The proof is similar to the proof of Theorem \ref{thm:main}
and so we shall be brief. One of the main differences is that here we let
\[
\rho_n:=\frac{\log \log n}{\log |\BA_n^1|},
\]
where $A_n=nA$ as before.
We remark that the choice of $\log \log n$ in the numerator is somewhat
arbitrary. Indeed, any function that goes to infinity sufficiently slow
as $n \to \infty$ would do. The purpose is to make sure that 
$|\BA_n^{\rho_n}|^{-\rho_n/600}$ vanishes in the limit and $\BA_n^{\rho_n}$
is good.

Observe also that by \eqref{eqn:scaleinv},
\[
\liminf_{n \to \infty} \frac{\log |\BA_n^1|}{\log n}
=\liminf_{n \to \infty} \frac{\log |\BA^{1/n}|}{-\log (1/n)}
\geq \liminf_{\rho \to 0} \frac{\log |\BA^\rho|}{-\log \rho}
=\underline{\dim}_B(A).
\]
Letting $\delta:=(\underline{\dim}_B(A)-\underline{\alpha})/2$ 
we then get that 
\begin{eqnarray*} 
\lefteqn{\left(\frac{\log \log n}{n\log n}\right)^{\underline{\dim}_B(A)-\delta}|\BA_n^{\rho_n}|}\\
& & =\left(\frac{\log |\BA_n^1|}{\log n}\right)^{\underline{\dim}_B(A)-\delta}
\left(\frac{\log \log n}{n\log |\BA_n^1|}\right)^{\underline{\dim}_B(A)-\delta}
|\BA_n^{\rho_n}| 
\nonumber \\
& & =\left(\frac{\log |\BA_n^1|}{\log n}\right)^{\underline{\dim}_B(A)-\delta}
\left(\frac{\rho_n}{n}\right)^{\underline{\dim}_B(A)-\delta}|\BA^{\rho_n/n}| 
\to \infty,
\end{eqnarray*}
by the definition of $\underline{\dim}_B(A)$ (see Section \ref{sec:dim}).

We then see that for $n$ larger than some $N$,
\begin{eqnarray*}
\lefteqn{\BP(\cT(A_n)
-\underline{\alpha}\log n \leq z)}\\
& & =\BP\Big(\cT(A_n)-\log |\BA_n^{\rho_n}|+\delta\log n+
(\underline{\dim}_B(A)-\delta)\log \log n
\nonumber\\
& & \ \ \ \ -\log((\log \log n)^{\underline{\dim}_B(A)-\delta})+\log \left(
\left(\frac{\log \log n}{n\log n}\right)^{\underline{\dim}_B(A)-\delta}
|\BA_n^{\rho_n}|\right) \leq  z\Bigg)\nonumber\\
& & \leq \BP\left(\cT(A_n)-\log |\BA_n^{\rho_n}|+\delta\log n
\leq z\right) \nonumber \\
& & \leq \BP\left(\cT(\BA_n^{\rho_n})-\log |\BA_n^{\rho_n}|
\leq z-\delta\log n\right). \nonumber
\end{eqnarray*}
where the last inequality follows since 
$\cT(A_n) \geq \cT(\BA_n^{\rho_n})$.
Using Theorem \ref{thm:fluct} and Lemma 
\ref{lem:sizeofA} part $a),$ we then get that 
\begin{eqnarray*}
\lefteqn{\BP(\cT(A_n)-\underline{\alpha}\log n \leq z)
\leq c_3|\BA_n^{\rho_n}|^{-\rho_n/600}+\exp(-e^{-z+\delta\log n}) }\\
& & 
\leq c_3|\BA_n^1|^{-\rho_n/600}+\exp(-e^{-z}n^\delta)
=c_3e^{-\log \log n/600}+\exp(-e^{-z}n^\delta) \to 0.
\end{eqnarray*}
The second statement is proved in the same way so we omit the proof.
\fbox{}\\

\section{Applications} \label{sec:applications}
The purpose of this section is two-fold. Firstly, we will demonstrate 
that any set containing a $d$-dimensional closed box satisfies
\eqref{eqn:dimlim}, and so there are many examples of sets for which 
our main result apply. Presumably, results such as 
Proposition \ref{prop:Arhoex} are well known, 
even though we could not find a reference for this exact statement.
Secondly we will consider examples of sets
where $\dim_B(A)<d$ and see what our main results imply for those sets.

\begin{prop} \label{prop:Arhoex}
Let $A=[0,1]^d.$ Then,
\[
0<\lim_{\rho \to 0} \rho^d |\BA^\rho|<\infty,
\]
and in particular the limit exists.
It follows that any bounded set $A$ such that 
$[x,x+\delta]^d\subset A$ for some $x\in \BR^d$ and $\delta>0,$ 
must satisfy \eqref{eqn:dimliminfsup}.
\end{prop}
\noindent
{\bf Proof.}
Let $c>0$ and note that $A=[0,1]^d$ is contained in the union of
$(\lceil (c\rho)^{-1}\rceil)^{d}$ translates of the set 
$A_{c\rho}=[0,c\rho]^d$. It follows from \eqref{eqn:scaleinv} that  
$|\BA_{c\rho}^\rho|=|\BA^{\rho/(c\rho)}|=|\BA^{1/c}|$.
Then, by using Lemma \ref{lem:sizeofA} part $b)$, we get that 
\begin{equation}\label{eq:limfinite}
|\BA^\rho| \leq (\lceil (c\rho)^{-1}\rceil)^{d}|\BA_{c\rho}^\rho|
\leq ((c\rho)^{-1}+1)^{d}|\BA^{1/c}|.
\end{equation}
Therefore,
\[
\limsup_{\rho \to 0}\rho^d |\BA^\rho|
\leq \limsup_{\rho \to 0}(c^{-1}+\rho)^{d}|\BA^{1/c}|
=c^{-d}|\BA^{1/c}|,
\]
and so 
\[
\limsup_{\rho \to 0}\rho^d |\BA^\rho|
\leq \liminf_{c \to \infty}c^{-d}|\BA^{1/c}|
=\liminf_{\rho \to 0}\rho^d |\BA^\rho|.
\]
This proves that the limit exists. 

We then observe that \eqref{eq:limfinite} 
immediately implies that the limit is finite. Furthermore, we must 
have that 
\[
|\BA^\rho|\geq |A\cap (\rho \BZ^d)|\geq (\rho^{-1}-1)^d,
\]
and so $\lim_{\rho \to 0}\rho^d|\BA^\rho|\geq 1.$

We now turn to the second statement. To that end, $A\subset \BR^d$ is now 
any bounded set that includes some box $B=[x,x+\delta]^d.$ 
Thus,
\[
\liminf_{\rho \to 0} \rho^d |\BA^\rho|
\geq \liminf_{\rho \to 0} \rho^d |\BB^\rho|>0,
\]
by the first statement of the proposition. Similarly, for large enough 
$\Delta$, we let $D=[y,y+\Delta]^d \supset A$ and then we find that 
\[
\limsup_{\rho \to 0} \rho^d |\BA^\rho|
\leq \limsup_{\rho \to 0} \rho^d |\BD^\rho|<\infty.
\]
\fbox{}\\

\medskip
\noindent
\begin{example}
Proposition \ref{prop:Arhoex} lets us apply Theorem \ref{thm:main} to 
the set $[0,1]^d.$ Informally, we then know that $\cT(n[0,1]^d)$ will be 
of order $d(\log n +\log \log n)$. 

A discrete analogue of $n[0,1]^d$ might be taken to be 
$\BD_n=[0,n-1]^d \cap \BZ^d$. 
However, Theorem \ref{thm:maindiscrete} implies that $\cT(\BD_n)$ will
be of order $\log |\BD_n|=d \log n$. Thus, the sequence 
$(\cT(\BD_n))_{n \geq 1}$ behaves differently from the sequence 
$(\cT(n[0,1]^d))_{n \geq 1}$ in that the first is ''missing'' the extra 
factor $d \log \log n.$
We point out that the results of \cite{Belius} are similar to the 
ones for $(\BD_n)_{n \geq 1},$ while the results of \cite{Jan2} are 
similar to when the ones for $(n[0,1]^d)_{n \geq 1}.$ This seems to indicate
that the extra factor $d \log \log n$ arises from the covered sets 
being non-discrete rather then from the fact that we are using unbounded
cylinders to perform our covering.

Of course, it is important to point out that the sequence 
$(\BD_n)_{n \geq 1}$ is not obtained by starting with some $B$ and then 
multiplying it by $n.$  
\end{example}

\medskip
\noindent
\begin{example}
In this example we consider a two-dimensional Cantor set. Since this 
is a well known
set we shall be somewhat informal in its description (see \cite{Falconer} 
Example 4.3 for details when $d=1$). We also remark that one can easily 
generalize this set into higher dimensions.
We start with the unit box $F_0=[0,1]^2,$ and in our first step 
we delete everything except the corner cubes of side length $1/3.$
Thus, we let 
$F_1=[0,1/3]^2\cup [0,1/3]\times[2/3,1]\cup[2/3,1]\times[0,1/3]\cup[2/3,1]^2.$
We then continue by repeating the exact same procedure on a smaller scale 
within each of the four retained sub-boxes in order to obtain $F_2.$
Continuing, we get a sequence $F_k \supset F_{k+1}\supset \cdots$ 
where $F_k$ consists of $4^k$ boxes of side length $3^{-k}.$
We then define 
\[
F:=\cap_{k=1}^\infty F_k,
\]
and it is easy to check that $\dim_B(F)=\frac{\log 4}{\log 3}.$ 

It is not hard to see that $\BF^{3^{-k}}$ must consist
of the union of all four corner points of the sub boxes of $\BF_k.$ Indeed, 
each sub box has side length $3^{-k}$ so that the distance between two 
adjacent corner points is exactly $3^{-k}$. Therefore, each sub box 
cannot contain more than these four points. We see that 
\[
\lim_{k \to \infty}3^{-k}|\BF^{3^{-k}}|
=\lim_{k \to \infty}3^{-k}4 3^k=4,
\]
and by interpolating we get that 
\[
1=\liminf_{\rho \to 0} \rho |\BF^\rho|
\leq \limsup_{\rho \to 0} \rho |\BF^\rho|=4.
\]
Therefore, Theorem \ref{thm:main} can be applied to show that the sequence 
\[
\left(\cT(nF)-\frac{\log 4}{\log 3}\left(\log n+\log \log n \right)\right)_{n \geq 1}
\]
is tight. Informally, this means that the cover time $\cT(nF)$ will be of
order $\frac{\log 4}{\log 3}\left(\log n+\log \log n \right)$ with some 
fluctuations.
\end{example}

\noindent
\begin{example}
Here we let $A=[0,1]^k\times\{0\}^{d-k} \subset \BR^d$. As in Proposition 
\ref{prop:Arhoex} it is easy to verify that Theorem \ref{thm:main} is 
applicable. Then, we conclude that for any dimensions $d,$ the sequence
\[
\left(\cT(nA)-k\left(\log n+\log \log n \right)\right)_{n \geq 1}
\]
is tight. However, the constants in our results are allowed to depend on 
$d$, so it is possible that a stronger or at least different result can be 
obtained by letting $d \to \infty$ at the same time as $n \to \infty$. 
However, we choose not to pursue this here. 
\end{example}

\noindent
\begin{example}
Our last example will be of a sequence of finite sets. To that end,
consider $\BI_n$ consisting of $\frac{n \log n}{\log \log n}$ equidistant
points on the interval $[0,n].$ Then, let 
$\BB_n=\BI_n^d=\BI_n\times \cdots \times \BI_n$. Clearly, 
\[
\Sep(\BB_n)\log |\BB_n|
=\frac{n}{\frac{n \log n}{\log \log n}} 
\log\left(\frac{n \log n}{\log \log n}\right)^d
=d \log \log n+O((\log n)^{-1/2})\to \infty,
\]
so that Theorem \ref{thm:maindiscrete} tells us that 
$\cT(\BB_n)-\log |\BB_n|=\cT(\BB_n)-d(\log n+\log \log n-\log \log \log n)$
converges to a Gumbel distributed random variable. 

Note that if $\BI_n$ would instead consist of $n \log n$ points, then 
Theorem \ref{thm:maindiscrete} would not be applicable. However, 
estimates can be obtained by using Theorem \ref{thm:fluct} on approximations
of the set. 
\end{example}

\noindent
{\bf Acknowledgement.} The authors would like to thank S. Janson for many useful 
comments and suggestions.

\end{document}